\documentclass[a4paper, leqno, 12pt]{article}

\title{\textsc{An arithmetic Riemann-Roch theorem for pointed stable curves}}
\author{\textsc{Gerard Freixas i Montplet}}
\date{}

\usepackage{amsmath, mathrsfs, amssymb, amscd, amsthm}
\usepackage[]{fontenc}
\usepackage{xy}
\usepackage{enumerate}
\xyoption{all}

\numberwithin{equation}{section}

\theoremstyle{plain}
\newtheorem{theorem}{Theorem}[section]
\newtheorem{proposition}[theorem]{Proposition}
\newtheorem{lemma}[theorem]{Lemma}
\newtheorem{corollary}[theorem]{Corollary}
\newtheorem*{theoremA}{Theorem A}
\newtheorem*{theoremB}{Theorem B}

\theoremstyle{definition}
\newtheorem{definition}[theorem]{Definition}
\newtheorem{construction}[theorem]{Construction}
\newtheorem{notation}[theorem]{Notation}

\theoremstyle{remark}
\newtheorem{remark}[theorem]{Remark}

\DeclareMathOperator{\Spec}{Spec}
\DeclareMathOperator{\hyp}{hyp}
\DeclareMathOperator{\pre}{pre}
\DeclareMathOperator{\ACH}{\widehat{CH}}
\DeclareMathOperator{\CH}{CH}
\DeclareMathOperator{\ac1}{\widehat{c}_{1}}
\DeclareMathOperator{\c1}{c_{1}}
\DeclareMathOperator{\PSL}{PSL}

\DeclareMathOperator{\an}{an}
\DeclareMathOperator{\Real}{Re}

\DeclareMathOperator{\adeg}{\widehat{deg}}
\DeclareMathOperator{\Gm}{\mathbb{G}_{m}}
\DeclareMathOperator{\Res}{Res}

\DeclareMathOperator{\reg}{reg}
\DeclareMathOperator{\sing}{sing}
\DeclareMathOperator{\gdiv}{div}
\DeclareMathOperator{\detp}{det^{\prime}}
\DeclareMathOperator{\cusps}{cusps}

\newcommand{\OO}{\mathcal{O}}
\newcommand{\XX}{\mathcal{X}}
\newcommand{\BS}{\mathcal{S}}
\newcommand{\UU}{\mathcal{U}}
\newcommand{\CC}{\mathbb{C}}
\newcommand{\Int}{\mathbb{Z}}
\newcommand{\SCM}{\mathcal{\overline{M}}}
\newcommand{\PP}{\mathbb{P}}
\newcommand{\HH}{\mathbb{H}}
\newcommand{\RR}{\mathbb{R}}
\newcommand{\QQ}{\mathbb{Q}}
\newcommand{\QQQ}{\overline{\mathbb{Q}}}
\newcommand{\pd}{\partial}
\newcommand{\cz}{\overline{z}}
\newcommand{\C}{\mathcal{C}}
\newcommand{\SCC}{\mathcal{\overline{C}}}
\newcommand{\SM}{\mathcal{M}}
\newcommand{\DD}{\mathcal{D}}
\newcommand{\ff}{\mathfrak{f}}
\newcommand{\fg}{\mathfrak{g}}
\newcommand{\fp}{\mathfrak{p}}
\newcommand{\fX}{\mathfrak{X}}
\newcommand{\fY}{\mathfrak{Y}}
\newcommand{\fZ}{\mathfrak{Z}}
\newcommand{\fD}{\mathfrak{D}}
\newcommand{\fL}{\mathfrak{L}}

\newcommand{\cpd}{\overline{\pd}}

\begin{document}
\setcounter{tocdepth}{1}
\setcounter{section}{0}
\maketitle

\footnotesize\textsc{Abstract.-} Let $(\OO, \Sigma, F_{\infty})$ be an arithmetic ring of Krull dimension at most 1, $\BS=\Spec\OO$ and $(\pi:\XX\rightarrow\BS; \sigma_{1},\ldots,\sigma_{n})$ a $n$-pointed stable curve of genus $g$. Write $\UU=\XX\setminus\cup_{j}\sigma_{j}(\BS)$. The invertible sheaf $\omega_{\XX/\BS}(\sigma_{1}+\ldots+\sigma_{n})$ inherits a hermitian structure $\|\cdot\|_{\hyp}$ from the dual of the hyperbolic metric on the Riemann surface $\UU_{\infty}$. In this article we prove an arithmetic Riemann-Roch type theorem that computes the arithmetic self-intersection of $\omega_{\XX/\BS}(\sigma_{1}+\ldots+\sigma_{n})_{\hyp}$. The theorem is applied to modular curves $X(\Gamma)$, $\Gamma=\Gamma_{0}(p)$ or $\Gamma_{1}(p)$, $p\geq 11$ prime, with sections given by the cusps. We show $Z^{\prime}(Y(\Gamma),1)\sim e^{a}\pi^{b}\Gamma_{2}(1/2)^{c}L(0,\SM_{\Gamma})$, with $p\equiv 11\mod 12$ when $\Gamma=\Gamma_{0}(p)$. Here $Z(Y(\Gamma),s)$ is the Selberg zeta function of the open modular curve $Y(\Gamma)$, $a,b,c$ are rational numbers, $\SM_{\Gamma}$ is a suitable Chow motive and $\sim$ means equality up to algebraic unit.\\

\textsc{R\'esum\'e.-} Soit $(\OO, \Sigma, F_{\infty})$ un anneau arithm\'etique de dimension de Krull au plus 1, $\BS=\Spec\OO$ et $(\pi:\XX\rightarrow\BS; \sigma_{1},\ldots,\sigma_{n})$ une courbe stable $n$-point\'ee de genre $g$. Posons $\UU=\XX\setminus\cup_{j}\sigma_{j}(\BS)$. Le faisceau inversible $\omega_{\XX/\BS}(\sigma_{1}+\ldots+\sigma_{n})$ h\'erite une structure hermitienne $\|\cdot\|_{\hyp}$ du dual de la m\'etrique hyperbolique sur la surface de Riemann $\UU_{\infty}$. Dans cet article nous prouvons un th\'eor\`eme de Riemann-Roch arithm\'etique qui calcule l'auto-intersection arithm\'etique de $\omega_{\XX/\BS}(\sigma_{1}+\ldots+\sigma_{n})_{\hyp}$. Le th\'eor\`eme est appliqu\'e aux courbes modulaires $X(\Gamma)$, $\Gamma=\Gamma_{0}(p)$ ou $\Gamma_{1}(p)$, $p\geq 11$ premier, prenant les cusps comme sections. Nous montrons $Z^{\prime}(Y(\Gamma),1)\sim e^{a}\pi^{b}\Gamma_{2}(1/2)^{c}L(0,\SM_{\Gamma})$, avec $p\equiv 11\mod 12$ lorsque $\Gamma=\Gamma_{0}(p)$. Ici $Z(Y(\Gamma),s)$ est la fonction zeta de Selberg de la courbe modulaire ouverte $Y(\Gamma)$, $a,b,c$ sont des nombres rationnels, $\SM_{\Gamma}$ est un motif de Chow appropri\'e et $\sim$ signifie \'egalit\'e \`a unit\'e pr\`es.

\normalsize

\tableofcontents

\section{Introduction}
Let $(\OO, \Sigma, F_{\infty})$ be an arithmetic ring of Krull dimension at most $1$ \cite[Def. 3.1.1]{GS}. This means that $\OO$ is an excellent, regular, Noetherian integral domain, $\Sigma$ is a finite non-empty set of monomorphisms $\sigma:\OO\hookrightarrow\CC$ and $F_{\infty}:\CC^{\Sigma}\rightarrow\CC^{\Sigma}$ is a conjugate-linear involution of $\CC$-algebras such that the diagram
\begin{displaymath}
	\xymatrix{
		&	&\CC^{\Sigma}\ar[dd]^{F^{\infty}}\\
		& \OO\ar[ru]^{\delta}\ar[rd]^{\delta}	&\\
		&	&\CC^{\Sigma}
	}
\end{displaymath}
commutes. Here $\delta$ is induced by the set $\Sigma$. Define $\BS=\Spec\OO$ and let $(\pi:\XX\rightarrow\BS;\sigma_{1},\ldots,\sigma_{n})$ be a $n$-pointed stable curve of genus $g$, in the sense of Knudsen and Mumford \cite[Def. 1.1]{Knudsen}. Assume that $\XX$ is regular. Write $\UU=\XX\setminus\cup_{j}\sigma_{j}(\BS)$. To $\XX$ and $\UU$ we associate the complex analytic spaces
\begin{displaymath}
 	\XX_{\infty}=\bigsqcup_{\sigma\in\Sigma}\XX_{\sigma}(\CC),\,\,\,\UU_{\infty}=\bigsqcup_{\sigma\in\Sigma}\UU_{\sigma}(\CC).
\end{displaymath}
Notice that $F_{\infty}$ acts on $\XX_{\infty}$ and $\UU_{\infty}$. The stability hypothesis guarantees that every connected component of $\UU_{\infty}$ has a hyperbolic metric of constant curvature $-1$. The whole family is invariant under the action of $F_{\infty}$. Dualizing we obtain an arakelovian --i.e. invariant under $F_{\infty}$-- hermitian structure $\|\cdot\|_{\hyp}$ on $\omega_{\XX/\BS}(\sigma_{1}+\ldots+\sigma_{n})$. Contrary to the requirements of classical Arakelov theory \cite{GS2}, the metric $\|\cdot\|_{\hyp}$ is \textit{not smooth}, but has some mild singularities of logarithmic type. Actually $\|\cdot\|_{\hyp}$ is a \textit{pre-log-log} hermitian metric in the sense of Burgos-Kramer-K\"uhn \cite[Sec. 7]{BKK}. Following loc. cit., there is a first arithmetic Chern class $\ac1(\omega_{\XX/\BS}(\sigma_{1}+\ldots+\sigma_{n})_{\hyp})$ that lives in a pre-log-log arithmetic Chow group $\ACH^{1}_{\pre}(\XX)$. The authors define an intersection product
\begin{displaymath}
	\ACH^{1}_{\pre}(\XX)\otimes_{\Int}\ACH^{1}_{\pre}(\XX)\overset{\cdot}{\longrightarrow}\ACH^{2}_{\pre}(\XX)
\end{displaymath}
and a pushforward map
\begin{displaymath}
	\pi_{\ast}:\ACH^{2}_{\pre}(\XX)\longrightarrow\ACH^{1}(\BS).
\end{displaymath}
This paper is concerned with the class $\pi_{\ast}(\ac1(\omega_{\XX/\BS}(\sigma_{1}+\ldots+\sigma_{n})_{\hyp})^{2})$.

In their celebrated work \cite{ARR}, Gillet and Soul\'e --with deep contributions of Bismut-- proved an arithmetic analogue of the Grothendieck-Riemann-Roch theorem. Their theorem deals with the push-forward of a smooth hermitian vector bundle by a proper and generically smooth morphism of arithmetic varieties. The associated relative complex tangent bundle is equipped with a \textit{smooth} K\"ahler structure. With the notations above, if $n=0$ and $g\geq 2$, then the metric $\|\cdot\|_{\hyp}$ is smooth and the arithmetic Grothendieck-Riemann-Roch theorem may be applied to $\omega_{\XX/\BS,\hyp}$ and the \textquotedblleft hyperbolic\textquotedblright\, K\"ahler structure on $\XX_{\infty}$. The result is a relation between $\pi_{*}(\ac1(\omega_{\XX/\BS,\hyp})^{2})\in\ACH^{1}(\BS)$ and the class $\ac1(\lambda(\omega_{\XX/\BS}) ,\|\cdot\|_{Q})$, where $\|\cdot\|_{Q}$ is the Quillen metric corresponding to our data. However, for $n>0$ the singularities of $\|\cdot\|_{\hyp}$ prevent from applying the theorem of Gillet and Soul\'e.

The present article focuses on the so far untreated case $n>0$. We prove an arithmetic analogue of the Riemann-Roch theorem that relates $\pi_{*}(\ac1(\omega_{\XX/\BS}(\sigma_1+\ldots+\sigma_n)_{\hyp})^{2})$ and $\ac1(\lambda(\omega_{\XX/\BS}),\|\cdot\|_{Q})$. The Quillen type metric $\|\cdot\|_{Q}$ is defined by means of the Selberg zeta function of the connected components of $\UU_{\infty}$ (see Definition \ref{definition:Quillen}). In contrast with the result of Gillet and Soul\'e, our formula involves the first arithmetic Chern class of a new hermitian line bundle $\psi_{W}$. The corresponding invertible sheaf is the pull-back of the so called \textit{tautological psi} line bundle on the moduli stack $\SCM_{g,n}$, by the classifying morphism $\BS\rightarrow\SCM_{g,n}$. The underlying hermitian structure is dual to Wolpert's renormalization of the hyperbolic metric \cite[Def. 1]{Wolpert:cusps} (see also Definition \ref{definition:Wolpert_metric} below). The class $\ac1(\psi_{W})$ reflects the appearance of the continuous spectrum in the spectral resolution of the hyperbolic laplacian. After the necessary normalizations and definitions given in Section \ref{section:conventions}, the main theorem is stated as follows:
\begin{theoremA}
Let $g,n\geq 0$ be integers with $2g-2+n>0$, $(\OO, \Sigma, F_{\infty})$ an arithmetic ring of Krull dimension at most $1$ and $\BS=\Spec\OO$. Let $(\pi:\XX\rightarrow\BS;\sigma_{1},\ldots,\sigma_{n})$ be a $n$-pointed stable curve of genus $g$, with $\XX$ regular. For every closed point $\wp\in\BS$ denote by $n_{\wp}$ the number of singular points in the geometric fiber $\XX_{\wp}$ and put $\Delta_{\XX/\BS}=[\sum_{\wp}n_{\wp}\wp]\in\CH^{1}(\BS)$. Then the identity
\begin{displaymath}
	\begin{split}
	12\ac1(\lambda(\omega_{\XX/\BS})_{Q})-\Delta_{\XX/\BS}+&\ac1(\psi_{W})=\\
	&\pi_{\ast}\left(\ac1(\omega_{\XX/\BS}(\sigma_1+\ldots+\sigma_n)_{\hyp})^{2}\right)\\
	&+\ac1\left(\OO(C(g,n))\right)
	\end{split}
\end{displaymath}
holds in the arithmetic Chow group $\ACH^{1}(\BS)$.
\end{theoremA}
The theorem is deduced from the Mumford isomorphism on $\SCM_{g,n}$ (Theorem \ref{theorem:Mumford_isomorphism}) and a metrized version that incorporates the appropriate hermitian structures (Theorem \ref{theorem:Mumford_isometry}).\footnote{In particular, with the formalism of \cite[Sec. 4.3]{BKK2}, the assumption of regularity of $\XX$ can be weakened to $\pi:\XX\rightarrow\BS$ generically smooth.} The techniques employed combine the geometry of the boundary of $\SCM_{g+n,0}$ --through the so called clutching morphisms-- and the behavior of the small eigenvalues of the hyperbolic laplacian on degenerating families of compact surfaces. By a theorem of Burger \cite[Th. 1.1]{Burger} we can replace the small eigenvalues by the lengths of the pinching geodesics. Then Wolpert's pinching expansion of the family hyperbolic metric \cite[Exp. 4.2]{Wolpert:hyperbolic} provides an expression of these lengths in terms of a local equation of the boundary divisor $\pd\SM_{g+n,0}$. This gives a geometric manner to treat the small eigenvalues.  Another consequence of theorems \ref{theorem:Mumford_isomorphism} and \ref{theorem:Mumford_isometry} is a significant case of the local index theorem of Takhtajan-Zograf \cite{ZT:ZT_metric_0}--\cite{ZT:ZT_metric} (Theorem \ref{theorem:ZT} below).

Natural candidates to which Theorem A applies are provided by arithmetic models of modular curves, taking their cusps as sections. We focus on the curves $X(\Gamma)/\CC$, where $\Gamma\subset\PSL_{2}(\Int)$ is a congruence subgroup of the kind $\Gamma_{0}(p)$ or $\Gamma_{1}(p)$. We assume that $p\geq 11$ is a prime number. If $\Gamma=\Gamma_{0}(p)$, we further suppose $p\equiv 11\mod 12$. This guarantees in particular that $X(\Gamma)$ has genus $g\geq 1$. To $X(\Gamma)$ we attach two kinds of zeta functions:
\begin{itemize}
	\item[--] let $Y(\Gamma):=X(\Gamma)\setminus\lbrace \text{cusps}\rbrace$ be the open modular curve. Then $Y(\Gamma)$ is a hyperbolic Riemann surface of finite type. We denote by $Z(Y(\Gamma),s)$ the Selberg zeta function of $Y(\Gamma)$ (see Section \ref{section:conventions}). It is a meromorphic function defined over $\CC$, with a simple zero at $s=1$;
	\item[--] let $\text{Prim}_{2}(\Gamma)$ be a basis of normalized Hecke eigenforms for $\Gamma$. To $f\in\text{Prim}_{2}(\Gamma)$ we can attach a Chow motive $\SM(f)$ over $\QQ$, with coefficients in a suitable finite extension $F$ of $\QQ(\mu_{p})$, independent of $f$.\footnote{The construction of $\SM(f)$ amounts to the decomposition of the jacobian $\text{Jac}(X(\Gamma))$ under the action of the Hecke algebra. More generally, Deligne \cite[Sec. 7]{Deligne_motifs} and Scholl \cite[Th. 1.2.4]{Scholl} associate a Grothendieck --i.e. homological-- motive to any normalized new Hecke eigenform of weight $k\geq 2$, level $n$ and character $\chi$.} If $\chi$ is a Dirichlet character with values in $F^{\times}$, we denote by $\QQ(\chi)$ its Artin motive. For instance we may take $\chi=\overline{\chi_{f}}$, for the Dirichlet character $\chi_{f}$ associated to $f\in\text{Prim}_{2}(\Gamma)$. If $\text{Sym}^{2}$ denotes the square symmetrization projector and $(2)$ the Tate twist by $2$, we put
	\begin{displaymath}
		\SM_{\Gamma}:=\bigoplus_{f\in\text{Prim}_{2}(\Gamma)}\text{Sym}^{2}\SM(f)\otimes\QQ(\overline{\chi_{f}})(2)\in\text{Ob}(M_{\text{rat}}(\QQ)_{F}).
	\end{displaymath}
	The motivic $L$-function of $\SM_{\Gamma}$, $L(s,\SM_{\Gamma})$, can be defined --with the appropriate definition of the local factor at $p$-- so that we have the relation
	\begin{equation}\label{equation:-1}
		L(s,\SM_{\Gamma})
		=\prod_{f\in\text{Prim}_{2}(\Gamma)}L(s+2,\text{Sym}^{2}f,\overline{\chi_{f}}).
	\end{equation}
	The reader is referred to \cite{Carayol}, \cite[Sec. 7]{Deligne_motifs}, \cite[Sec. 5]{Hida}, \cite{Scholl} and \cite{Shimura} for details.\footnote{The factors $L(s+2,\text{Sym}^{2}f,\overline{\chi_{f}})$ where already studied by Hida \cite{Hida}, Shimura \cite{Shimura} and Sturm \cite{Sturm}.}
\end{itemize}
Denote by $\Gamma_{2}$ the Barnes double Gamma function \cite{Barnes} (see also \cite{Sarnak} and \cite{Voros}).
\begin{theoremB}
Let $p\geq 11$ be a prime number and $\Gamma=\Gamma_{0}(p)$ or $\Gamma_{1}(p)$. Assume $p\equiv 11\mod 12$ whenever $\Gamma=\Gamma_{0}(p)$. Then there exist rational numbers $a,b,c$ such that
\begin{displaymath}
	Z^{\prime}(Y(\Gamma),1)\sim_{\QQQ^{\times}}e^{a}\pi^{b}\Gamma_{2}(1/2)^{c}L(0,\SM_{\Gamma}),
\end{displaymath}
where $\alpha\sim_{\QQQ^{\times}}\beta$ means $\alpha=q\beta$ for some $q\in \QQQ^{\times}$.\footnote{The exponents $a,b,c$ can actually be computed in terms of $p$.}
\end{theoremB}
The proof relies on Theorem A and the computation of Bost \cite{Bost} and K\"uhn \cite{Kuhn} for the arithmetic self-intersection number of $\omega_{X_{1}(p)/\QQ(\mu_{p})}(\text{cusps})_{\hyp}$. Undertaking the proof of Bost and K\"uhn --under the form of Rohrlich's modular version of Jensen's formula \cite{Rohrlich}-- and applying Theorem A to $(\PP^{1}_{\Int};0,1,\infty)$, one can also show the equality
\begin{equation}\label{equation:-1bis}
	Z^{\prime}(\Gamma(2),1)=4\pi^{5/3}\Gamma_{2}(1/2)^{-8/3},
\end{equation}
where $Z(\Gamma(2),s)$ is the Selberg zeta function of the congruence group $\Gamma(2)$. The details are given in our thesis \cite[Ch. 8]{GFM:thesis}. However, our method fails to provide the exact value of $Z^{\prime}(\PSL_{2}(\Int),1)$.

To the knowledge of the author, the special values $Z^{\prime}(Y(\Gamma),1)$ remained unknown. Even though it was expected that they encode interesting arithmetic information (see \cite{Hashimoto} and \cite{Sarnak2}), it is quite remarkable that they can be expressed in terms of the special values $L(0,\SM_{\Gamma})$. The introduction of $\SM_{\Gamma}$ in the formulation of the theorem was suggested by Beilinson's conjectures (see \cite{Soule} for an account) and two questions of Fried \cite[Sec. 4, p. 537 and App., Par. 4]{Fried}. Fried asks about the number theoretic content of the special values of Ruelle's zeta function and an interpretation in terms of regulators.\footnote{The Ruelle zeta function $R(s)$ of a hyperbolic Riemann surface is related to the Selberg zeta function $Z(s)$ by $R(s)=Z(s)/Z(s+1)$. For instance, $R^{\prime}(1)=Z^{\prime}(1)/Z(2)$.} Also Theorem B may be seen as an analogue of the product formula for number fields $\prod_{\nu}|x|_{\nu}=1$. This analogy alone deserves further study.

So far there have been other attempts of proof of Theorem A. This is the case of \cite[Part II]{Weng}. The method of loc. cit. seems to lead to an analogous statement up to an unknown universal constant. The advantage of our approach is that explicit computations --such as Theorem B and (\ref{equation:-1bis})-- are allowed. Moreover, in contrast with \cite[Fund. rel. IV', p. 280]{Weng}, our result is available for pointed stable curves of any genus.\footnote{The proof of loc. cit. presents a gap in genus $g\leq 2$ (last two lines in page 279). The case $g=2$ requires a justification whereas there are counterexamples to the principle of the proof in genus $0$ and $1$: there exist non-constant harmonic functions on $\SM_{g,n;\CC}$, for $g=0$, $n\geq 4$ and $g=1$, $n\geq 1$.  Also the case $g=0$ and $n=3$ is beyond the reach in \cite{Weng}.} Remark \ref{remark:psi_cont} below strengthens the importance of the case $g=0$ and $n=3$.

In his forthcoming thesis \cite{Hahn}, T. Hahn obtains different results related to Theorem A. His approach is much in the spirit of Jorgenson--Lundelius \cite{JL1}--\cite{JL3}. In contrast with our geometric considerations, Hahn works with a degenerating family of metrics on a fixed compact Riemann surface and studies the behavior of the corresponding family of heat kernels. Consequences are derived for the family of heat trace regularizations and spectral zeta functions. It is likely that the two approaches can be combined to produce more general statements.

Let us briefly describe the structure of this paper.

Section \ref{section:conventions} fixes the normalizations to be followed throughout the paper. Specially we define the Wolpert and Quillen metrics occurring in the statement of Theorem A. In Section \ref{section:mumford} we review the definition and properties of the clutching morphisms of Knudsen. We also recall the Mumford isomorphism on $\SCM_{g,0}$. With these tools at hand, we show how to derive the Mumford isomorphsim on $\SCM_{g,n}$ (Theorem \ref{theorem:Mumford_isomorphism}). Sections \ref{section:analytic1} and \ref{section:analytic2} are devoted to the analytic part of the proof of the main theorem. In Section \ref{section:analytic1} we introduce the Liouville metric on the first tautological line bundle and establish its continuity and behavior under pull-back by the clutching morphisms. In Section \ref{section:analytic2} we recall the results of Wolpert on the degeneration of the Selberg zeta function for degenerating families of compact hyperbolic Riemann surfaces \cite{Wolpert:Selberg}. We also review the theorem of Burger on the small eigenvalues of such  families of curves \cite{Burger}. We derive consequences for the Selberg zeta function as well as for the Quillen metric. In Section \ref{section:Mumford_isometry} we prove a metrized version of the Mumford isomorphism on $\SM_{g,n}$ (Theorem \ref{theorem:Mumford_isometry}). The strategy relies on sections \ref{section:mumford}--\ref{section:analytic2} and the arithmetic Riemann-Roch theorem of Gillet and Soul\'e. Theorem A is then deduced as an immediate application of Theorem \ref{theorem:Mumford_isomorphism} and Theorem \ref{theorem:Mumford_isometry}. The article finishes with the proof of Theorem B in Section \ref{section:TheoremB}.
\section{Conventions and notations}\label{section:conventions}
We fix some conventions and notations that will hold throughout this paper.

Let $g,n\geq 0$ be integers with $2g-2+n>0$. We define the real constants
\begin{displaymath}
	\begin{split}
		C(g,n)=&\exp\left((2g-2+n)\left(\frac{\zeta^{\prime}(-1)}{\zeta(-1)}+\frac{1}{2}\right)\right),\\
		E(g,n)=&2^{(g+2-n)/3}\pi^{-n/2}\\
		&\cdot\exp\left((2g-2+n)\left(2\zeta^{\prime}(-1)-\frac{1}{4}+\frac{1}{2}\log(2\pi)\right)\right),
	\end{split}
\end{displaymath}
where $\zeta$ denotes the Riemann zeta function. Notice the relations
\begin{align}
	&C(g+n,0)=C(g,n)C(1,1)^{n},\label{equation:C}\\
	&E(g+n,0)=\pi^{n}E(g,n)E(1,1)^{n}.\label{equation:E}
\end{align}

Let $X$ be a compact and connected Riemann surface of genus $g$ and $p_{1}$,$\ldots$, $p_{n}$ distinct points in $X$. The open subset $U=X\setminus\lbrace p_{1},\ldots,p_{n}\rbrace$ admits a complete hyperbolic riemannian metric, of constant curvature -1. Denote it by $ds_{\hyp,U}^{2}$. Via a fuchsian uniformization $U\simeq\Gamma\backslash\HH$, $\Gamma\subset\PSL_{2}(\RR)$ torsion free, the metric $ds_{\hyp,U}^{2}$ is obtained by descent from the $\Gamma$ invariant Riemann tensor on $\HH$
\begin{displaymath}
	ds_{\hyp,\HH}^{2}=\frac{dx^{2}+dy^{2}}{y^{2}},\,\,\,z=x+iy\in\HH.
\end{displaymath}
Associated to $ds_{\hyp,U}^{2}$ there is a hermitian metric on the complex line $T_{U}$, that we write $h_{U}$. It is obtained by descent from the metric $h_{\HH}$ on $T_{\HH}$ defined by the rule
\begin{displaymath}
	h_{\HH}\left(\frac{\pd}{\pd z},\frac{\pd}{\pd z}\right)=\frac{1}{2y^{2}}.
\end{displaymath}
The hermitian metric $h_{U}$ extends to a pre-log-log hermitian metric $\|\cdot\|_{\hyp}$ on $\omega_{X}(p_{1}+\ldots+p_{n})$ \cite[Sec. 7.3.2]{GFM}. The first Chern form of $\omega_{X}(p_{1}+\ldots+p_{n})_{\hyp}$, which is defined on $U$, coincides with the normalized K\"ahler form $\omega$ of $h_{U}$ (curvature $-1$ condition). The form $\omega$ is locally given by
\begin{displaymath}
	\omega=\frac{i}{2\pi}h_{U}\left(\frac{\pd}{\pd z},\frac{\pd}{\pd z}\right) dz\wedge d\cz.
\end{displaymath}
The volume of $X$ with respect to $\omega$ is $2g-2+n$.

For every puncture $p_{j}$ there is a conformal coordinate $z$ with $z(p_{j})=0$, by means of which a small punctured disc $D^{*}(0,\varepsilon)\subset\CC$ with the Poincar\'e metric
\begin{displaymath}
	ds_{P}^{2}=\left(\frac{|dz|}{|z|\log|z|}\right)^{2}
\end{displaymath}
isometrically embeds into $(U,ds_{\hyp,U}^{2})$. Such a coordinate is unique up to rotation and is called a \textit{rs} coordinate at the cusp $p_{j}$.
\begin{definition}[Wolpert metric \cite{Wolpert:cusps}, Def. 1]\label{definition:Wolpert_metric}
Let $z$ be a \textit{rs} coordinate at the cusp $p_{j}$. The \textit{Wolpert metric} on the complex line $\omega_{X,p_{j}}$ is defined by
\begin{displaymath}
	\|dz\|_{W,p_{j}}=1.
\end{displaymath}
The tensor product $\otimes_{j}\omega_{X,p_{j}}$ is equipped with the tensor product of Wolpert metrics, and we write $\|\cdot\|_{W}$ for the resulting metric.
\end{definition}
The complex vector space $\C^{\infty}(X,\omega_{X})(\supset H^{0}(X,\omega_{X}))$ is equipped with the non-degenerate hermitian form
\begin{displaymath}
	\langle\alpha,\beta\rangle_{0}=\frac{i}{2\pi}\int_{X}\alpha\wedge\overline{\beta}.
\end{displaymath}
The space $H^{1}(X,\omega_{X})^{\vee}$ is canonically isomorphic to $H^{0}(X,\OO_{X})$ $=\CC$ via the analytic Serre duality. Since $\omega$ is integrable, the $L^{2}$ metric on $\C^{\infty}(X,\OO_{X})(\supset H^{0}(X,\OO_{X}))$ with respect to $h_{U}$ is well defined. If $\mathbf{1}$ is the function with constant value 1, then
\begin{displaymath}
	\langle\mathbf{1},\mathbf{1}\rangle_{1}=\int_{X}\omega=2g-2+n.
\end{displaymath}
The complex line $\lambda(\omega_{X})=\det H^{0}(X,\omega_{X})\otimes\det H^{1}(X,\omega_{X})^{-1}$ is endowed with the determinant metric build up from $\langle\cdot,\cdot\rangle_{0}$ and $\langle\cdot,\cdot\rangle_{1}$. We refer to it by $\|\cdot\|_{L^{2}}$.

We next recall the definition of the Selberg zeta function of $U$ (see \cite{Hejhal}). For every real $l>0$ the function
\begin{displaymath}
	Z_{l}(s)=\prod_{k=1}^{\infty}(1-e^{-(s+k)l})^{2}
\end{displaymath}
is holomorphic in $\Real s>0$. In a first step, the Selberg zeta function of $U$ is defined by the absolutely convergent product
\begin{displaymath}
	Z(U,s)=\prod_{\gamma}Z_{l(\gamma)}(s),\,\,\,\Real s>1,
\end{displaymath}
running over the simple closed non-oriented geodesics of the hyperbolic surface $(U,ds_{\hyp,U}^{2})$. Then one shows that $Z(U,s)$ extends to a meromorphic function on $\CC$, with a simple zero at $s=1$.
\begin{definition}[Quillen metric]\label{definition:Quillen}
We define the \textit{Quillen metric} on $\lambda(\omega_{X})$, attached to the hyperbolic metric on $U$, to be
\begin{displaymath}
	\|\cdot\|_{Q}=(E(g,n)Z^{\prime}(U,1))^{-1/2}\|\cdot\|_{L^{2}}.
\end{displaymath}
\end{definition}
We denote by $\OO(C(g,n))$ the trivial line bundle equipped with the norm $C(g,n)|\cdot|$, where $|\cdot|$ stands for the usual absolute value.

Let $(\OO, \Sigma, F_{\infty})$ be an arithmetic ring of Krull dimension at most $1$. Put $\BS=\Spec\OO$ and denote its generic point by $\eta$. Let $(\pi:\XX\rightarrow\BS;\sigma_{1},\ldots,\sigma_{n})$ be a $n$-pointed stable curve of genus $g$, in the sense of Knudsen and Mumford \cite[Def. 1.1]{Knudsen}. Assume $\XX_{\eta}$ regular. We write $\UU=\XX\setminus\cup_{j}\sigma_{j}(\BS)$. By definition $\XX_{\eta}$ is geometrically connected. For every complex embedding $\tau\in\Sigma$, the preceding constructions apply to $\UU_{\tau}(\CC)\subset\XX_{\tau}(\CC)$. Varying $\tau$, we obtain arakelovian hermitian line bundles $\omega_{\XX/\BS}(\sigma_{1}+\ldots+\sigma_{n})_{\hyp}$, $\psi_{W}=\otimes_{j}\sigma_{j}^{\ast}(\omega_{\XX/\BS})_{W}$, $\lambda(\omega_{\XX/\BS})_{Q}$ and $\OO(C(g,n))$. Similar notations will be employed for the analogous constructions over more general bases. In this way, we shall consider the \textquotedblleft universal situation\textquotedblright\, $(\pi:\SCC_{g,n}\rightarrow\SCM_{g,n};\sigma_{1},\ldots,\sigma_{n})$, where $\SCM_{g,n}$ is the Deligne-Mumford stack of $n$-pointed stable curves of genus $g$, $\SCC_{g,n}$ the universal family and $\sigma_{1},\ldots,\sigma_{n}$ the universal sections. We then have \textquotedblleft universal hermitian line bundles\textquotedblright\, $\lambda(\omega_{\SCC_{g,n}/\SCM_{g,n}})\mid_{\SM_{g,n};Q}$, $\sigma^{\ast}_{j}(\omega_{\SCC_{g,n}/\SCM_{g,n}})\mid_{\SM_{g,n};W}$, etc. ($\SM_{g,n}$ is the open substack of smooth curves). When the context is clear enough, we freely write $\lambda_{g,n;Q}$, $\sigma_{j}^{\ast}(\omega_{\SCC_{g,n}/\SCM_{g,n}})_{W}$, $\psi_{g,n;W}$, etc.

If $F$ is an algebraic stack of finite type over $\Spec\Int$, then we denote by $F^{\an}$  the analytic stack associated to $F_{\CC}$. For instance, applied to $\SM_{g,n}$ and $\SCM_{g,n}$, we obtain the analytic stack $\SM_{g,n}^{\an}$ of $n$-punctured Riemann surfaces of genus $g$ and its Deligne-Mumford stable compactification $\SCM_{g,n}^{\an}$. If $\mathcal{E}_{1}\rightarrow\mathcal{E}_{2}$ is a morphism of sheaves over $F$, there is an associated morphism $\mathcal{E}_{1}^{\an}\rightarrow\mathcal{E}_{2}^{\an}$ over $F^{\an}$. Finally, to a morphism $F\rightarrow G$ between algebraic stacks of finite type over $\Spec\Int$, corresponds a morphism between analytic stacks $F^{\an}\rightarrow G^{\an}$.

Our standard references for the theory of algebraic stacks are \cite{DM} and \cite{LMB}. For pointed stable curves we refer to the original article of Knudsen \cite{Knudsen}. Concerning the arithmetic intersection theory, we follow the extension of the theory of Gillet-Soul\'e \cite{GS} developed by Burgos, Kramer and K\"uhn \cite{BKK}. 
\section{Knudsen's clutching and the Mumford isomorphism on $\SCM_{g,n}$}\label{section:mumford}
\subsection{Preliminaries}
Let $g,n\geq 0$ be integers with $2g-2+n>0$. We denote by $\SCM_{g,n}\rightarrow\Spec\Int$ the Deligne-Mumford stack of $n$-pointed stable curves of genus $g$, and by $\pi:\SCC_{g,n}\rightarrow\SCM_{g,n}$ the universal curve \cite{Knudsen}. The morphism $\pi$ has $n$ universal sections, $\sigma_{1},\ldots,\sigma_{n}$. The first theorem compiles some geometric features of $\SCM_{g,n}$.
\begin{theorem}\label{theorem:irreducible_fibers}
i. $\SCM_{g,n}$ is a proper and smooth algebraic stack over $\Spec\Int$ of relative dimension $3g-3+n$. The substack $\pd\SM_{g,n}$ classifying singular curves is a divisor with normal crossings, relative to $\Spec\Int$.

ii. The fibers of  $\SCM_{g,n}$ over $\Spec\Int$ are geometrically irreducible.	
\end{theorem}
\begin{proof}
The first assertion is \cite[Th. 2.7]{Knudsen}. The second assertion is due to Deligne-Mumford \cite{DM} in the case $n=0$ and $g\geq 2$, to Deligne-Rapoport \cite{DR} in the case $n=g=1$ and Keel \cite{Keel} in the case $g=0$ and $n\geq 3$. For the remaining cases we proceed by induction. Assume the claim for some couple $(g,n)$, with $2g-2+n>0$. We derive the claim for $(g,n+1)$. Recall that $\SCC_{g,n}$ gets identified with $\SCM_{g,n+1}$ via Knudsen's contraction morphsim $c:\SCM_{g,n+1}\rightarrow\SCC_{g,n}$ \cite[Sec. 2]{Knudsen}. It then suffices to show that $\SCC_{g,n}$ has geometrically irreducible fibers over $\Spec\Int$. If $\SM_{g,n}$ is the dense open substack of $\SCM_{g,n}$ classifying non-singular curves, then $\pi^{-1}(\SM_{g,n})\rightarrow\SM_{g,n}$ is proper, smooth with geometrically connected fibers (by definition of pointed stable curve \cite[Def. 1.1]{Knudsen}). It follows that $\pi^{-1}(\SM_{g,n})$ has geometrically irreducible fibers over $\Spec\Int$. To conclude, we notice that $\pi^{-1}(\pd\SM_{g,n})\cup\sigma_{1}(\SCM_{g,n})\cup\ldots\cup\sigma_{n}(\SCM_{g,n})$ is a divisor with normal crossings relative to $\Spec\Int$. Therefore $\pi^{-1}(\SM_{g,n})$ is fiberwise dense in $\SCC_{g,n}$, over $\Spec\Int$. The proof is complete.
\end{proof}
\begin{corollary}\label{corollary:irreducible_fibers}
Let $N\geq 1$ be an integer. For an algebraic stack of the form $\SM:=\SCM_{g_{1},n_{1}}\times\ldots\times\SCM_{g_{r},n_{r}}$ we have
\begin{displaymath}
	H^{0}(\SM\times\Spec\Int[1/N],\Gm)=\Int[1/N]^{\times}.
\end{displaymath}
\end{corollary}
\begin{proof}
By Theorem \ref{theorem:irreducible_fibers}, $\SM\times\Spec\Int[1/N]$ is a proper and smooth algebraic stack over $\Spec\Int[1/N]$, with geometrically irreducible fibers. The corollary follows.
\end{proof}
We now define the tautological line bundles on $\SCM_{g,n}$. 
\begin{definition}[Tautological line bundles \cite{Knudsen}, \cite{Mumford:enumerative}] 
The \textit{tautological line bundles} on $\SCM_{g,n}$ are
\begin{displaymath}
	\begin{split}
		&\lambda_{g,n}=\lambda(\omega_{\SCC_{g,n}/\SCM_{g,n}})=\det(R\pi_{\ast}\omega_{\SCC_{g,n}/\SCM_{g,n}}),\\
		&\delta_{g,n}=\OO(\pd\SM_{g,n}),\\
		&\psi_{g,n}^{(j)}=\sigma_{j}^{\ast}\omega_{\SCC_{g,n}/\SCM_{g,n}},\,\,\,j=1,\ldots,n,\\
		&\psi_{g,n}=\otimes_{j}\psi_{g,n}^{(j)},\\
		&\kappa_{g,n}=\langle\omega_{\SCC_{g,n}/\SCM_{g,n}}(\sigma_{1}+\ldots+\sigma_{n}),\omega_{\SCC_{g,n}/\SCM_{g,n}}(\sigma_{1}+\ldots+\sigma_{n})\rangle,
	\end{split}
\end{displaymath}
where $\langle\cdot,\cdot\rangle$ denotes the Deligne pairing \cite[XVIII]{SGA4}, \cite{Elkik}.
\end{definition}
\subsection{Clutching morphisms}
We proceed to recall Knudsen's clutching morphism \cite[Part II, Sec. 3]{Knudsen}. The basic construction is resumed in the following theorem.
\begin{theorem}[Knudsen]\label{theorem:clutching}
Let $\pi:X\rightarrow S$ be a pre-stable curve (i.e. flat, proper, whose fibers are geometrically connected with at worst ordinary double points). Let $\sigma_{1},\sigma_{2}:S\rightarrow X$ be given disjoint sections of $\pi$. Suppose that $\pi$ is smooth along $\sigma_{1},\sigma_{2}$. Then there is a diagram
\begin{displaymath}
	\xymatrix{
		& X\ar[r]^{p}\ar[d]^{\pi}
		& X^{\prime}\ar[d]_{\pi^{\prime}}\\
		&S\ar@{=}[r]\ar@/^/[u]^{\sigma_{i}}
		&S\ar@/_/[u]_{\sigma}
	}
\end{displaymath}
such that:

i. $\sigma=p\sigma_{1}=p\sigma_{2}$ and $p$ is universal with this property;

ii. $p$ is finite;

iii. $\pi^{\prime}:X^{\prime}\rightarrow S$ is a prestable curve, fiberwise obtained by identification of $\sigma_{1}$ and $\sigma_{2}$ in an ordinary double point;

iv. for every open $U\subseteq X^{\prime}$ we have
\begin{displaymath}
	H^{0}(U,\OO_{X^{\prime}})=\lbrace f\in H^{0}(p^{-1}(U),\OO_{X})\mid\sigma_{1}^{\ast}(f)=\sigma_{2}^{\ast}(f)\rbrace;
\end{displaymath}

v. let $\psi^{(j)}=\sigma_{j}^{\ast}\omega_{X/S}$, $j=1,2$. There is an exact sequence
\begin{equation}\label{equation:1}
	0\rightarrow\sigma_{\ast}(\psi^{(1)}\otimes\psi^{(2)})\overset{\alpha}{\rightarrow}\Omega_{X^{\prime}/S}\rightarrow p_{*}\Omega_{X/S}\rightarrow 0.
\end{equation}
The arrow $\alpha$ of (\ref{equation:1}) is defined as follows. Let $\mathcal{I}_{j}$ (resp. $\mathcal{I}$) be the ideal sheaf of the image of $\sigma_{j}$ (resp. $\sigma$) in $X$ (resp. $X^{\prime}$), $j=1,2$. Consider the natural morphisms $p_{j}:\mathcal{I}/\mathcal{I}^{2}\rightarrow p_{\ast}(\mathcal{I}_{j}/\mathcal{I}_{j}^{2})=\sigma_{\ast}\psi^{(j)}$. There is a natural isomorphism
\begin{equation}\label{equation:2}
	\begin{split}
		\bigwedge^{\hspace{0.5cm}2}\mathcal{I}/\mathcal{I}^{2}&\longrightarrow\sigma_{\ast}(\psi^{(1)}\otimes\psi^{(2)})\\
		u\wedge v&\longmapsto p_{1}u\otimes p_{2}v-p_{1}v\otimes p_{2}u.
	\end{split}
\end{equation}
Via the isomorphism (\ref{equation:2}), $\alpha$ gets identified with $u\wedge v\mapsto udv$.
\end{theorem}
The clutching morphism is defined by a three step construction.
\begin{construction}[Clutching morphism]\label{construction:clutching}
\hspace{0cm}\\
\textit{Step 1.} Let $\pi_{1}:X_{1}\rightarrow S$ be a $n_{1}+1$-pointed stable curve of genus $g_{1}$, with sections $\sigma_{1}^{(1)},\ldots,\sigma_{n_{1}+1}^{(1)}$. In addition consider the 3-pointed stable curve of genus 0, $(\pi:\PP^{1}_{S}\rightarrow S;0,1,\infty)$. By Theorem \ref{theorem:clutching} we can attach $\PP^{1}_{S}$ to $X_{1}$ by identification of the sections $1$ and $\sigma_{n_{1}+1}^{(1)}$. We obtain a new pointed stable curve $X_{1}^{\prime}$ of genus $g_{1}$, with sections $\sigma_{1}^{(1)},\ldots,\sigma_{n_{1}}^{(1)}$ and $0_{1}$, $\infty_{1}$. We proceed analogously with a $n_{2}+1$ pointed stable curve $\pi_{2}:X_{2}\rightarrow S$, with sections $\sigma_{1}^{(2)},\ldots,\sigma_{n_{2}+1}^{(2)}$.\\
\textit{Step 2.} We apply Theorem \ref{theorem:clutching} to glue $X_{1}^{\prime}$ and $X_{2}^{\prime}$ along the sections $\infty_{1}$ and $\infty_{2}$. We obtain a new pointed stable curve $X$ of genus $g_{1}+g_{2}$ with sections $\sigma_{1}^{(1)},\ldots,\sigma_{n_{1}}^{(1)}$, $\sigma_{1}^{(2)},\ldots,\sigma_{n_{2}}^{(2)}$, $0_{1}$, $0_{2}$.\\
\textit{Step 3.} We contract the sections $0_{1}$ and $0_{2}$ \cite[Prop. 2.1]{Knudsen}. We obtain a $n_{1}+n_{2}$ pointed stable curve of genus $g_{1}+g_{2}$.
\end{construction}
\begin{theorem}[Knudsen]\label{theorem:clutching2}
Construction \ref{construction:clutching} defines a morphism of algebraic stacks
\begin{displaymath}
	\beta:\SCM_{g_{1},n_{1}+1}\times\SCM_{g_{2},n_{2}+1}\longrightarrow\SCM_{g_{1}+g_{2},n_{1}+n_{2}},
\end{displaymath}
which is representable, finite and unramified. If moreover $g_{1}\neq g_{2}$ or $n_{1}+n_{2}\neq 0$, then $\beta$ is a closed immersion. 
\end{theorem}
\begin{proof}
This is \cite[Th. 3.7 and Cor. 3.9]{Knudsen}.
\end{proof}
The following statement describes the behavior of the tautological line bundles under pull-back by the clutching morphism.
\begin{proposition}\label{proposition:clutching}
Let $g=g_{1}+g_{2}$, $n=n_{1}+n_{2}$. There are isomorphisms of line bundles, uniquely determined up to a sign,
\begin{align}
		&\beta^{\ast}\lambda_{g,n}\overset{\sim}{\rightarrow}\lambda_{g_{1},n_{1}+1}\boxtimes\lambda_{g_{2},n_{2}+1},
		\label{equation:3}\\
		&\beta^{\ast}\delta_{g,n}\overset{\sim}{\rightarrow}(\delta_{g_{1},n_{1}+1}\otimes\psi^{(n_{1}+1)\,-1}_{g_{1},n_{1}+1})
		\boxtimes(\delta_{g_{2},n_{2}+1}\otimes\psi^{(n_{2}+1)\,-1}_{g_{2},n_{2}+1}),\label{equation:4}\\
		&\beta^{\ast}\psi_{g,n}\overset{\sim}{\rightarrow}(\psi_{g_{1},n_{1}+1}\otimes\psi_{g_{1},n_{1}+1}^{(n_{1}+1)\,-1})
		\boxtimes(\psi_{g_{2},n_{2}+1}\otimes\psi_{g_{2},n_{2}+1}^{(n_{2}+1)\,-1}),\label{equation:5}\\
		&\beta^{\ast}\kappa_{g,n}\overset{\sim}{\rightarrow}\kappa_{g_{1},n_{1}+1}\boxtimes\kappa_{g_{2},n_{2}+1}.\label{equation:6}
\end{align}
\end{proposition}
\begin{proof}
Once the existence of (\ref{equation:3})--(\ref{equation:6}) is proven, the uniqueness assertion already follows from Corollary \ref{corollary:irreducible_fibers}. 

For the isomorphisms (\ref{equation:3})--(\ref{equation:5}) we refer to \cite[Th. 4.3]{Knudsen} (they are easily constructed by means of Theorem \ref{theorem:clutching} above). We now focus on (\ref{equation:6}). The formation of the relative dualizing sheaf is compatible with base change \cite[Sec. 1]{Knudsen}, as well as for the Deligne pairing \cite[Par. I.3, p. 202]{Elkik}. By the definition of the clutching morphism (Construction \ref{construction:clutching}), we first reduce to the following situation. Let $S$ be a noetherian integral scheme and $(\pi_{i}:X_{i}\rightarrow S;\sigma_{1}^{(i)},\ldots,\sigma_{n_{i}+1}^{(i)})$, $i=1,2$, two pointed stable curves of genus $g_{i}$, respectively. We apply Theorem \ref{theorem:clutching} to the pre-stable curve $X=X_{1}\sqcup X_{2}\rightarrow S$ with sections $\sigma_{n_{1}+1}^{(1)}$, $\sigma_{n_{2}+1}^{(2)}$. With the notations of the theorem, we have to construct a natural isomorphism
\begin{align}
	&\left\langle \omega_{X^{\prime}/S}(\sum_{j=1}^{n_{1}}\sigma_{j}^{(1)}+\sum_{j=1}^{n_{2}}\sigma_{j}^{(2)}), \omega_{X^{\prime}/S}(\sum_{j=1}^{n_{1}}\sigma_{j}^{(1)}+\sum_{j=1}^{n_{2}}\sigma_{j}^{(2)})\right\rangle
	\overset{\sim}{\longrightarrow}\label{equation:7}\\
	&\hspace{4cm}\left\langle\omega_{X_{1}/S}(\sum_{j=1}^{n_{1}+1}\sigma_{j}^{(1)}),\omega_{X_{1}/S}(\sum_{j=1}^{n_{1}+1}\sigma_{j}^{(1)})\right\rangle\label{equation:8}\\
	&\hspace{4cm}\otimes\left\langle\omega_{X_{2}/S}(\sum_{j=1}^{n_{2}+1}\sigma_{j}^{(2)}),\omega_{X_{2}/S}(\sum_{j=1}^{n_{2}+1}\sigma_{j}^{(2)})\right\rangle.\label{equation:9}
\end{align}
We denote by $L$ the line bundle of (\ref{equation:7}) and $L_{1}$, $L_{2}$ the line bundles of (\ref{equation:8}), (\ref{equation:9}), respectively. After localizing for the \'etale topology on $S$, we can find rational sections $s,t$ of $\omega_{X^{\prime}/S}(\sum_{j=1}^{n_{1}}\sigma_{j}^{(1)}+\sum_{j=1}^{n_{2}}\sigma_{j}^{(2)})$ whose divisors are finite and flat over $S$, with mutually disjoint components, disjoint with the image of the section $\sigma$ (image of $\sigma_{n_{1}+1}^{(1)}$ and $\sigma_{n_{2}+1}^{(2)}$ in $X^{\prime}$). Denote by $s\mid_{X_{1}}$, $t\mid_{X_{1}}$ (resp. $s\mid_{X_{2}}$, $t\mid_{X_{2}}$) the pull-backs of $s$ and $t$ to $X_{1}$ (resp. $X_{2}$), respectively. By the properties of the relative dualizing sheaf, $s\mid_{X_{1}}$, $t\mid_{X_{1}}$ (resp. $s\mid_{X_{2}}$, $t\mid_{X_{2}}$) are rational sections of $\omega_{X_{1}/S}(\sum_{j=1}^{n_{1}+1}\sigma_{j}^{(1)})$ (resp. $\omega_{X_{2}/S}(\sum_{j=1}^{n_{2}+1}\sigma_{j}^{(2)})$). They have finite and flat divisors over $S$, with mutually disjoint components and disjoint from $\sigma_{n_{1}+1}^{(1)}$ (resp. $\sigma_{n_{2}+1}^{(2)}$). The symbols $\langle s,t\rangle$, $\langle s\mid_{X_{1}},t\mid_{X_{1}}\rangle$ and $\langle s\mid_{X_{2}},t\mid_{X_2}\rangle$ are non-zero sections of $L$, $L_1$ and $L_2$, respectively. We define the assignment
\begin{displaymath}
	\Phi:\langle s,t\rangle\longmapsto\langle s\mid_{X_{1}},t\mid_{X_{1}}\rangle\otimes\langle s\mid_{X_2},t\mid_{X_2}\rangle.
\end{displaymath}
The symbols of the form $\langle s,t\rangle$ generate $L$. Clearly $\Phi$ defines a morphism $L\rightarrow L_{1}\otimes L_{2}$, compatible with base change by noetherian integral schemes. Observe that $\Phi$ is injective. Indeed, with the assumptions and notations above, $\langle s\mid_{X_{1}},t\mid_{X_{1}}\rangle\otimes\langle s\mid_{X_{2}}, t\mid_{X_{2}}\rangle$ is non-zero by integrality of $S$. We now prove that $\Phi$ is an isomorphism. By Nakayama's lemma, we reduce to $S=\Spec k$, where $k$ is an algebraically closed field. In this case, by definition of the relative dualizing sheaf, we have
\begin{displaymath}
	\omega_{X^{\prime}/S}(\sum_{j=1}^{n_{1}}\sigma_{j}^{(1)}+\sum_{j=1}^{n_{2}}\sigma_{j}^{(2)})\mid_{X_{i}}=\omega_{X_{i}/S}(\sum_{j=1}^{n_{i}+1}\sigma_{j}^{(i)}),\,\,\,i=1,2.
\end{displaymath}
For $i=1,2$, let $s_{i},t_{i}$ be rational sections of $\omega_{X_{i}/S}(\sum_{j=1}^{n_{i}+1}\sigma_{j}^{(i)})$, whose divisors have mutually disjoint components, disjoint from $\sigma_{n_{i}+1}^{(i)}$. Consider $s_{i}, t_{i}$ as sections of $\omega_{X_{i}/S}$. We can define $a(s_{i})=\Res_{\sigma_{n_{i}+1}^{(i)}} s_{i}\in k^{\times}$, $b(t_{i})=\Res_{\sigma_{n_{i}+1}^{(i)}}t_{i}\in k^{\times}$, $i=1,2$. We introduce the sections $s$, $t$ of $\omega_{X^{\prime}/S}(\sum_{j=1}^{n_{1}}\sigma_{j}^{(1)}+\sum_{j=1}^{n_{2}}\sigma_{j}^{(2)})$ characterized by
\begin{displaymath}
	s=\begin{cases}
		a(s_{2})s_{1} &\text{on }\,\,\, X_{1},\\
		-a(s_{1})s_{2} &\text{on }\,\,\, X_{2},
	\end{cases}
\end{displaymath}
and
\begin{displaymath}
	t=\begin{cases}
		b(t_{2})t_{1} &\text{on }\,\,\, X_{1},\\
		-b(t_{1})t_{2} &\text{on }\,\,\, X_{2}.
	\end{cases}
\end{displaymath}
The divisors of $s$ and $t$ have mutually disjoint components, disjoint from $\sigma$. We compute
\begin{displaymath}
	\Phi\langle s,t\rangle=a(s_{1})a(s_{2})b(t_{1})b(t_{2})\langle s_{1},t_{1}\rangle\otimes\langle s_{2},t_{2}\rangle.
\end{displaymath}
Since $\langle s_{1},t_{1}\rangle\otimes\langle s_{2},t_{2}\rangle$ is a frame of $L_{1}\otimes L_{2}$ and $a(s_{1})a(s_{2})b(t_{1})b(t_{2})\in k^{\times}$, we conclude that $\Phi$ is surjective. 

Notice that the construction of $\Phi$ naturally extends to a base $S$ equal to an arbitrary disjoint union of noetherian integral schemes, in particular to objects of the \'etale site of $\SCM_{g_{1},n_{1}+1}\times\SCM_{g_{2},n_{2}+1}$. Applying the functoriality of the Deligne pairing and the relative dualizing sheaf, it is easily checked that $\Phi$ is compatible with base change on this site. Therefore it descends to the required isomorphism. The proof of the proposition is complete.
\end{proof}
\begin{corollary}\label{corollary:clutching}
Let $\gamma:\SCM_{g,n}\times\SCM_{1,1}^{\times n}\rightarrow\SCM_{g+n,0}$ be obtained by reiterated applications of clutching morphisms. Then we have isomorphisms, uniquely determined up to a sign,
\begin{align}
	&\gamma^{\ast}\lambda_{g+n,0}\overset{\sim}{\rightarrow}\lambda_{g,n}\boxtimes\lambda_{1,1}
	^{\boxtimes n},\label{equation:3bis}\\
	&\gamma^{\ast}\delta_{g+n,0}\overset{\sim}{\rightarrow}(\delta_{g,n}\otimes\psi_{g,n}^{-1})\boxtimes
	(\delta_{1,1}\otimes\psi_{1,1}^{-1})^{\boxtimes n},\label{equation:4bis}\\
	&\gamma^{\ast}\kappa_{g+n,0}
	\overset{\sim}{\rightarrow}\kappa_{g,n}\boxtimes\kappa_{1,1}^{\boxtimes n}.\label{equation:6bis}
\end{align}
\end{corollary}
\begin{proof}
This is a straightforward application of Proposition \ref{proposition:clutching}.
\end{proof}
\begin{remark}
The isomorphisms (\ref{equation:3bis})--(\ref{equation:6bis}) are described, locally for the \'etale topology, by means of Theorem \ref{theorem:clutching} and the proof of Proposition \ref{proposition:clutching} (see also Knudsen \cite[Part III, Sec. 4]{Knudsen}).
\end{remark}
\subsection{The Mumford isomorphism on $\SCM_{g,n}$}
The next theorem generalizes to $\SCM_{g,n}$ the so called Mumford isomorphism on $\SCM_{g,0}$. The statement is known for $\SCM_{g,n}$ over a field \cite[Eq. 3.15, p. 109]{ArbarelloCornalba}. Our approach is well suited for the analytic part of the proof of the main theorem. The idea is based on two points: a) pull-back the Mumford isomorphism on $\SCM_{g+n,0}$ by the clutching morphism of Corollary \ref{corollary:clutching}; b) deduce the Mumford isomorphism on $\SCM_{g,n}$ from point a) and the Mumford isomorphism on $\SCM_{1,1}$.
\begin{theorem}\label{theorem:Mumford_isomorphism}
There is an isomorphism of line bundles on $\SCM_{g,n}/\Int$, unique\-ly determined up to a sign,
\begin{displaymath}
	\DD_{g,n}:\lambda_{g,n}^{\otimes 12}\otimes\delta_{g,n}^{-1}\otimes\psi_{g,n}\overset{\sim}{\longrightarrow}\kappa_{g,n}.
\end{displaymath}
\end{theorem}
\begin{proof}
For the cases $g\geq 2$, $n=0$ and $g=n=1$, we refer to Moret-Bailly \cite[Th. 2.1]{MB} (which is based on Mumford \cite{Mumford} and Deligne-Rapoport \cite{DR}). For the latter, a comment is in order: \cite{MB} provides an isomorphism
\begin{equation}\label{equation:9_1}
	\lambda_{1,1}^{\otimes 12}\otimes\delta_{1,1}^{-1}\overset{\sim}{\longrightarrow}\langle\omega_{\SCC_{1,1}/\SCM_{1,1}},
	\omega_{\SCC_{1,1}/\SCM_{1,1}}\rangle.
\end{equation}
Since $\pi:\SCC_{1,1}\rightarrow\SCM_{1,1}$ is smooth along the universal section $\sigma_{1}$, we have the adjunction isomorphism
\begin{displaymath}
	\langle\OO(\sigma_{1}),\OO(\sigma_{1})\rangle\overset{\sim}{\longrightarrow}\langle\omega_{\SCC_{1,1}/\SCM_{1,1}},\OO(\sigma_{1})\rangle^{-1},
\end{displaymath}
and hence an isomorphism
\begin{equation}\label{equation:9_2}
	\langle\omega_{\SCC_{1,1}/\SCM_{1,1}},\omega_{\SCC_{1,1}/\SCM_{1,1}}\rangle\overset{\sim}{\longrightarrow}
	\kappa_{1,1}\otimes\psi_{1,1}^{-1}.
\end{equation}
The isomorphism $\DD_{1,1}$ is then constructed with (\ref{equation:9_1})--(\ref{equation:9_2}). For the general case, we first claim that there is an isomorphism,
\begin{equation}\label{equation:10}
	\DD_{g,n}^{\prime}:pr_{1}^{\ast}(\lambda_{g,n}^{\otimes 12}\otimes\delta_{g,n}^{-1}\otimes\psi_{g,n})
	\overset{\sim}{\longrightarrow}pr_{1}^{\ast}\kappa_{g,n}
\end{equation}
($pr_{1}$ is the projection onto the first factor).
Indeed, consider the clutching morphism $\gamma:\SCM_{g,n}\times\SCM_{1,1}^{\times n}\rightarrow\SCM_{g+n,0}$. From Corollary \ref{corollary:clutching} we deduce
\begin{equation}\label{equation:11}
	\begin{split}
	&\gamma^{\ast}\DD_{g+n,0}:(\lambda_{g,n}^{\otimes 12}\otimes\delta_{g,n}^{-1}\otimes\psi_{g,n})\boxtimes
	(\lambda_{1,1}^{\otimes 12}\otimes\delta_{1,1}^{-1}\otimes\psi_{1,1})^{\boxtimes n}\overset{\sim}{\longrightarrow}\\
	&\hspace{9cm}\kappa_{g,n}\boxtimes\kappa_{1,1}^{\boxtimes n}.
	\end{split}
\end{equation}
The claim follows tensoring (\ref{equation:11}) by $pr_{2}^{\ast}(\DD_{1,1}^{\boxtimes n})^{\otimes -1}$.

Let $p_{1}\neq p_{2}$ be prime numbers and $X_{1}\rightarrow\Spec\Int[1/p_{1}]$, $X_{2}\rightarrow\Spec\Int[1/p_{2}]$ two smooth $1$-pointed stable curves of genus $1$.\footnote{For instance, for $p=11, 17$ the modular curve $X_{0}(p)$ can be given the structure of an elliptic curve over $\QQ$ with good reduction over $\Spec\Int[1/p]$ \cite{DR}.} Associated to the curves $X_{j}$, $j=1,2$, there are morphisms
\begin{displaymath}
	\varphi_{j}:\Spec\Int[1/p_{j}]\longrightarrow\SCM_{1,1}\times\Spec\Int[1/p_{j}]\overset{\Delta}{\longrightarrow}\SCM_{1,1}^{\times n}\times\Spec\Int[1/p_{j}],
\end{displaymath}
where $\Delta$ is the $n$-diagonal map. Pulling $\DD_{g,n}^{\prime}$ back by $pr_{2}^{\ast}\varphi_{j}$, we obtain an isomorphism of line bundles on $\SCM_{g,n}\times\Spec\Int[1/p_{j}]$
\begin{displaymath}
	\DD^{(j)}_{g,n}:\lambda_{g,n}^{\otimes 12}\otimes\delta_{g,n}^{-1}\otimes\psi_{g,n}\otimes\Int[1/p_{j}]
	\overset{\sim}{\longrightarrow}\kappa_{g,n}\otimes\Int[1/p_{j}].
\end{displaymath}
Define the line bundle $L=(\lambda_{g,n}^{\otimes 12}\otimes\delta_{g,n}^{-1}\otimes\psi_{g,n})^{-1}\otimes\kappa_{g,n}$ on $\SCM_{g,n}$. Then $\DD^{(j)}_{g,n}$ induces a trivialization $\tau_{j}$ of $L\otimes\Int[1/p_{j}]$. Over the open subset $\Spec\Int[1/p_{1}p_{2}]$ of $\Spec\Int$, the trivializations $\tau_{1}$, $\tau_{2}$ differ by a unit in $\Int[1/p_{1}p_{2}]^{\times}=\pm p_{1}^{\Int}p_{2}^{\Int}$ (Corollary \ref{corollary:irreducible_fibers}). Write $\tau_{1}=\epsilon p_{1}^{m} p_{2}^{n}\tau_{2}$, over $\Spec\Int[1/p_{1}p_{2}]$, with $\epsilon\in\lbrace -1,1\rbrace$. Then $\tau_{1}^{\prime}:=\epsilon p_{1}^{-m}\tau_{1}$, $\tau_{2}^{\prime}:=p_{2}^{n}\tau_{2}$ are new trivializations of $L\otimes\Spec\Int[1/p_{1}]$, $L\otimes\Spec\Int[1/p_{2}]$, respectively. By construction, they coincide over $\Spec\Int[1/p_{1}]\cap\Spec\Int[1/p_{2}]$. Therefore they glue into a trivialization of $L$ over $\Spec\Int$. This establishes the existence of an isomorphism as in the statement. The uniqueness follows from Corollary \ref{corollary:irreducible_fibers}.
\end{proof}
\begin{corollary}\label{corollary:Mumford_isomorphism}
For the clutching morphism $\gamma:\SCM_{g,n}\times\SCM_{1,1}^{\times n}\rightarrow\SCM_{g+n,0}$, the isomorphisms $\gamma^{\ast}\DD_{g+n,0}$ and $\DD_{g,n}\boxtimes\DD_{1,1}^{\boxtimes n}$ coincide up to a sign.
\end{corollary}
\begin{proof}
The corollary is a straightforward application of Theorem \ref{theorem:Mumford_isomorphism}, Corollary \ref{corollary:clutching} and Corollary \ref{corollary:irreducible_fibers}.
\end{proof}
\section{The Liouville metric on $\kappa_{g,n}$}\label{section:analytic1}
\subsection{Local description of $\SCM_{g,n}^{\an}$}\label{subsection:local_description}
Let $(X;p_{1},\ldots,p_{n})$ be a pointed stable curve over $\CC$. Teichm\"uller's theory provides a \textit{small stable deformation} $(\mathfrak{f}:\mathfrak{X}\rightarrow\Omega;\fp_{1},\ldots,\fp_{n})$ of $(X^{\an};p_{1},\ldots,p_{n})$, where $\Omega\subset\CC^{3g-3+n}$ is some open analytic neighborhood of $0$. Let $\mathfrak{F}:\Omega\rightarrow\SCM_{g,n}^{\an}$ be the induced morphism of analytic stacks. After possibly shrinking $\Omega$, the image $\mathfrak{F}(\Omega)$ is an open substack of $\SCM_{g,n}^{\an}$. It is the stack theoretic quotient of $\Omega$ by a finite group. Varying $(X;p_{1},\ldots,p_{n})$ in $\SCM_{g,n}(\CC)$, the open subsets $\Omega$ as above cover $\SCM_{g,n}^{\an}$. This subsection is based on \cite[Sec. 2]{Masur} and \cite[Sec. 2]{Wolpert:hyperbolic}, and reviews the construction of the small stable deformation $(\mathfrak{f}:\mathfrak{X}\rightarrow\Omega; \fp_{1},\ldots, \fp_{n})$.

\begin{construction}\label{construction:deformation}
\textit{i}. Fix $(X; p_{1},\ldots, p_{n})$ a $n$-pointed stable curve of genus $g$ over $\CC$. We shall identify $X$ with $X^{\an}$ by Chow's theorem. Let $q_{1},\ldots,q_{m}$ be the singular points of $X$. Define $X^{\circ}:=X\setminus\lbrace p_{1},\ldots,p_{n},q_{1},\ldots,q_{m}\rbrace$ and $\overline{X}$ a smooth completion of $X^{\circ}$. Then $\overline{X}$ has a pair of punctures $a_{j}$, $b_{j}$ at the place of each cusp $q_{j}$. The surface $X^{\circ}$ has a unique complete riemannian metric of constant curvature $-1$. Let $(W_{i}, w_{i})$ be a \textit{rs} analytic chart at $p_{i}$ and $(U_{j}, u_{j})$, $(V_{j},v_{j})$ \textit{rs} analytic charts at $a_{j}$, $b_{j}$ respectively (see Section \ref{section:conventions}). We suppose that all the $W_{i}$, $U_{j}$, $V_{k}$ are mutually disjoint. Finally, let $U_{0}$ be a relatively compact open subset of $\overline{X}\setminus(\overline{\cup_{i,j}W_{i}\cup U_{j}\cup V_{j}})$.
 
\textit{ii}. The deformation space $\mathfrak{D}$ of $X^{\circ}$ is the product of the Teichm\"uller spaces of the connected components of $X^{\circ}$. We can choose smooth Beltrami differentials $\nu_{1},\ldots,\nu_{r}$ compactly supported in $U_{0}$, spanning the tangent space at $X^{\circ}$ of $\mathfrak{D}$. Let $s\in\CC^{r}$ and define $\nu(s)=\sum_{j}s_{j}\nu_{j}$. Then $\nu(s)$ is a smooth Beltrami differential compactly supported in $U_{0}$. For $|s|$ small, $\|\nu(s)\|_{\infty}<1$. The solution of the Beltrami equation on $X^{\circ}$, for the Beltrami differential $\nu(s)$, produces a new Riemann surface $\fX^{\circ}_{s}$ diffeomorphic to $X^{\circ}$. For this, let $\lbrace(U_{\alpha},z_{\alpha})\rbrace_{\alpha}$ be an analytic atlas of $X^{\circ}$. For every $\alpha$ we take $w_{\alpha}$ a homeomorphism solution of $w_{\overline{z}_{\alpha}}=\nu(s)w_{z_{\alpha}}$. We can normalize $w_{\alpha}$ to depend holomorphically on $s$. Then $\lbrace (U_{\alpha},w_{\alpha}\circ z_{\alpha})\rbrace$ is an atlas for the Riemann surface $\fX^{\circ}_{s}$. Notice that $w_{\alpha}$ is in general not holomorphic in $z_{\alpha}$, but quasi-conformal. However $w_{\alpha}\circ z_{\alpha}$ is holomorphic on $U_{\alpha}\cap (\cup_{i,j}W_{i}\cup U_{j}\cup V_{j})$, since $\nu(s)$ is supported in $U_{0}$. This means that $(U_{\alpha}\cap (\cup_{i,j}W_{i}\cup U_{j}\cup V_{j}), z_{\alpha})$ is an analytic chart of $\fX^{\circ}_{s}$. In particular, $(W_{i}\setminus\lbrace p_{i}\rbrace,w_{i})$, $(U_{j}\setminus\lbrace a_{j}\rbrace,u_{j})$ and $(V_{j}\setminus\lbrace b_{j}\rbrace, v_{j})$ serve as analytic charts on $\fX^{\circ}_{s}$.

\textit{iii}. Let $0<c<1$ be a small real number such that $u_{j}(U_{j})$ and $v_{j}(V_{j})$ contain the open disc $D(0,c)\subset\CC$, for all $j$. Let $t_{j}\in\CC$ with $|t_{j}|<c^{2}$. We remove from $\fX^{\circ}_{s}$ the discs $\lbrace |u_{j}|\leq |t_{j}|/c\rbrace\subset U_{j}$ and $\lbrace |v_{j}|\leq |t_{j}|/c\rbrace\subset V_{j}$. We obtain a Riemann surface $\fX^{\ast}_{s,t}$, $t=(t_{1},\ldots,t_{m})$. Then we form the identification space $\fX^{\circ}_{s,t}=\fX^{\ast}_{s,t}/\sim$, where $p\sim q$ is the equivalence relation generated by
\begin{displaymath}
	\begin{split}
		&p\in\lbrace |t_{j}|/c<|u_{j}|<c\rbrace,\,\,\,q\in\lbrace |t_{j}|/c<|v_{j}|<c\rbrace\,\,\,\text{for some}\,\,\,j,\\
		& p\sim q\,\,\,\text{if, and only if,}\,\,\, u_{j}(p)v_{j}(q)=t_{j}.
	\end{split}
\end{displaymath}
The analytic space $\fX_{s,t}^{\circ}$ is actually a Riemann surface. We can let some of the $t_{j}=0$, with the obvious meaning for $\fX_{s,t}^{\circ}$. Observe that the $(W_{i}\setminus\lbrace p_{i}\rbrace,w_{i})$, $i=1,\ldots,n$, still define analytic charts on $\fX_{s,t}^{\circ}$. Then the charts $(W_{i},w_{i})$ describe a completion $\fX_{s,t}$ of $\fX_{s,t}^{\circ}$, smooth near the points $\fp_{j}(s,t):=(w_{j}=0)$, $j=1,\ldots,m$. The tuple $(\fX_{s,t}; \fp_{1}(s,t),\ldots,\fp_{n}(s,t))$ is a $n$-pointed stable curve of genus $g$.

\textit{iv}. For $(s,t)$ in some small open analytic subset $\Omega\subset\CC^{3g-3+n}$ we have constructed a $n$-pointed stable curve of genus $g$. We put $\fX:=\cup_{(s,t)\in\Omega}\fX_{s,t}$ and $\ff:\fX\rightarrow\Omega$ for the natural projection.
\end{construction}
\begin{proposition}
i. The tuple $(\ff:\fX\rightarrow\Omega;\fp_{1},\ldots,\fp_{n})$ is a $n$-pointed stable curve of genus $g$, whose fiber at $0$ equals $(X;p_{1},\ldots,p_{n})$. Let $\mathfrak{F}:\Omega\rightarrow\SCM_{g,n}^{\an}$ be the induced morphism of analytic stacks. 

ii. After possibly shrinking $\Omega$, $\mathfrak{F}$ is a local manifold cover: the image $\mathfrak{F}(\Omega)$ is an open substack of $\SCM_{g,n}^{\an}$ and is the stack theoretic quotient of $\Omega$ by a finite group acting on $\Omega$.
\end{proposition}
\begin{proof}
We refer to \cite[Sec. 2]{Wolpert:hyperbolic}.
\end{proof}
\subsection{Definition of the Liouville metric}\label{subsection:definition_liouville}
\begin{construction}\label{construction:liouville}
Let $p\in\SCM_{g,n}(\CC)$ be a point corresponding to a pointed stable curve $(X; p_{1},\ldots, p_{n})$. Let $X_{1},\ldots,X_{m}$ be the decomposition of $X_{\reg}$ into connected components. Every Riemann surface $X_{j}^{\circ}=X_{j}\setminus\lbrace p_{1},\ldots,p_{n}\rbrace$ admits a complete hyperbolic metric of constant curvature $-1$, $ds_{\hyp,j}^{2}$. If $\overline{X}_{j}$ is a compactification of $X_{j}^{\circ}$ and $\pd X_{j}^{\circ}=\overline{X}_{j}\setminus X_{j}^{\circ}$, then $ds_{\hyp,j}^{2}$ induces a pre-log-log hermitian metric $\|\cdot\|_{\hyp,j}$ on $\omega_{\overline{X}_{j}}(\pd X_{j}^{\circ})$. For its first Chern form we write $\c1(\omega_{X_{j}}(\pd X_{j}^{\circ})_{\hyp})$ (well defined and smooth on $X_{j}^{\circ}$). If $\sigma,\tau$ are rational sections of $\omega_{X}(p_{1}+\ldots+p_{n})$, whose divisors have mutually disjoint components, disjoint from $X_{\sing}$, then the integral
\begin{displaymath}
	I_{j}(\sigma,\tau):=\int_{X_{j}^{\circ}}\left[\log\|\sigma\|_{\hyp,j}\delta_{\gdiv t}+\log\|\tau\|_{\hyp,j}\c1(\omega_{\overline{X}_{j}}(\pd X_{j}^{\circ})_{\hyp})\right]
\end{displaymath}
is convergent \cite[Sec. 7]{BKK}. The norm $\|\langle \sigma,\tau\rangle\|$ of $\langle \sigma,\tau\rangle$ is characterized by
\begin{displaymath}
	\log\|\langle \sigma,\tau\rangle\|=\sum_{j=1}^{m}I_{j}(\sigma,\tau).
\end{displaymath}
This construction defines a hermitian metric (at the archimedian places) on the tautological line bundle $\kappa_{g,n}$.
\end{construction}
\begin{definition}[Liouville metric]
The hermitian metric on $\kappa_{g,n}$ defined by Construction \ref{construction:liouville} is called the \textit{Liouville metric}. We write $\overline{\kappa}_{g,n}$ to refer to the line bundle $\kappa_{g,n}$ together with the Liouville metric.
\end{definition}
\begin{lemma}
Let $\|\cdot\|_{\hyp}$ be the hermitian metric on $\omega_{\SCC_{g,n}/\SCM_{g,n}}(\sigma_{1}+\ldots+\sigma_{n})$, induced by the hyperbolic metric on (the regular locus) of the fibers of $\SCC_{g,n}\setminus\cup_{j}\sigma_{j}(\SCC_{g,n})\rightarrow\SCM_{g,n}$. Then we have
\begin{equation}\label{equation:11bis}
	\overline{\kappa}_{g,n}=\langle\omega_{\SCC_{g,n}/\SCM_{g,n}}(\sigma_{1}+\ldots+\sigma_{n})_{\hyp}, \omega_{\SCC_{g,n}/\SCM_{g,n}}(\sigma_{1}+\ldots+\sigma_{n})_{\hyp}\rangle,
\end{equation}
where the right hand side of (\ref{equation:11bis}) is endowed with the Deligne metric \cite[Sec. 6]{Deligne}.
\end{lemma}
\begin{proof}
This is a reformulation of the definition of the Liouville metric.
\end{proof}
\begin{lemma}\label{lemma:clutching_liouville}
Let $g=g_{1}+g_{2}$, $n=n_{1}+n_{2}$ and $\beta:\SCM_{g_{1},n_{1}+1}\times\SCM_{g_{2},n_{2}+1}\rightarrow\SCM_{g,n}$ be Knudsen's clutching morphism. Then the isomorphism (\ref{equation:6}) induces an isometry 
\begin{displaymath}
	\beta^{\ast}\overline{\kappa}_{g,n}\overset{\sim}{\longrightarrow}\overline{\kappa}_{g_{1},n_{1}+1}\boxtimes\overline{\kappa}_{g_{2},n_{2}+1}.
\end{displaymath}
In particular, for the clutching morphism $\gamma:\SCM_{g,n}\times\SCM_{1,1}^{\times n}\rightarrow\SCM_{g+n,0}$ of Corollary \ref{corollary:clutching}, we have the isometry
\begin{displaymath}
	\gamma^{\ast}\overline{\kappa}_{g+n,0}\overset{\sim}{\longrightarrow}\overline{\kappa}_{g,n}\boxtimes\overline{\kappa}_{1,1}^{\boxtimes n}.
\end{displaymath}
\end{lemma}
\begin{proof}
One easily checks that the isomorphism (\ref{equation:6}), constructed in the proof of Proposition \ref{proposition:clutching}, is compatible with the Liouville metric.
\end{proof}
The main result of this section is the following theorem.
\begin{theorem}\label{theorem:liouville}
The Liouville metric is continuous on $\SCM_{g,0}^{\an}$.
\end{theorem}
We postpone the proof of Theorem \ref{theorem:liouville} until \textsection \ref{subsection:proof_theorem_liouville}. For the moment we derive a consequence of the theorem.
\begin{corollary}\label{corollary:liouville}
The Liouville metric is continuous on $\SCM_{g,n}^{\an}$.
\end{corollary}
\begin{proof}
We first observe that the metric on $\overline{\kappa}_{g,n}\boxtimes\overline{\kappa}_{1,1}^{\boxtimes n}$ is continuous. Indeed, consider the clutching morphism $\gamma:\SCM_{g,n}\times\SCM_{1,1}^{\times n}\rightarrow\SCM_{g+n,0}$. By Lemma \ref{lemma:clutching_liouville}, we have an isometry
\begin{displaymath}
	\gamma^{\ast}\overline{\kappa}_{g+n,0}\overset{\sim}{\longrightarrow}\overline{\kappa}_{g,n}\boxtimes\overline{\kappa}_{1,1}^{\boxtimes n}.
\end{displaymath}
The claim already follows from Theorem \ref{theorem:liouville}. Notice that, applied to the particular case $g=n=1$, this implies that the metric on $\overline{\kappa}_{1,1}\boxtimes\overline{\kappa}_{1,1}$ is continuous.

Let $\Delta:\SCM_{1,1}\rightarrow\SCM_{1,1}\times\SCM_{1,1}$ be the diagonal morphism. Then we have an isometry
\begin{displaymath}
	\Delta^{\ast}(\overline{\kappa}_{1,1}\boxtimes\overline{\kappa}_{1,1})\overset{\sim}{\longrightarrow}\overline{\kappa}_{1,1}^{\otimes 2}.
\end{displaymath}
We deduce that the metric on $\overline{\kappa}_{1,1}$ is continuous. Together with the continuity of the metric on $\overline{\kappa}_{g,n}\boxtimes\overline{\kappa}_{1,1}^{\boxtimes n}$, this shows that $pr_{1}^{\ast}\overline{\kappa}_{g,n}$ is a continuous hermitian line bundle. Hence so does $\overline{\kappa}_{g,n}$. The proof is complete.
\end{proof}
\begin{remark}
\textit{i}. By means of Teichm\"uller theory it can be shown that the Liouville metric is actually smooth on $\SM_{g,n}^{\an}$.

\textit{ii}. In \cite[Ch. 6]{GFM:thesis} we show that the Liouville metric is pre-log-log along $\pd\SM_{g,n}^{\an}$.

\textit{iii}. The name of Liouville metric is inspired by the Liouville action of Takhtajan-Zograf on $\SM_{0,n}^{\an}$ \cite{ZT:Liouville}.
\end{remark}
\subsection{Proof of Theorem \ref{theorem:liouville}}\label{subsection:proof_theorem_liouville}
The proof of Theorem \ref{theorem:liouville} is based on the next statement and the pinching expansion for the family hyperbolic metric established by Wolpert \cite[Exp. 4.2, p. 445]{Wolpert:hyperbolic}.
\begin{proposition}[Masur \cite{Masur}, Sec. 6, Eq. 6.6]\label{proposition:Masur}
Let $\ff:\fX\rightarrow\Omega$ be a small stable deformation in compact curves of genus $g\geq 2$, as in Construction \ref{construction:deformation}. For the hyperbolic metric $ds_{\hyp;s,t}^{2}$ on $\fX_{s,t}$ and for all $j=1,\ldots,m$ we write
\begin{displaymath}
	\begin{split}
	& ds_{\hyp;s,t}^{2}=\rho_{s,t}(u_{j})\left(\frac{|du_{j}|}{|u_{j}|\log|u_{j}|}\right)^{2},\,\,\,\text{on}\,\,\,A_{j}(t)=\lbrace |t_{j}|^{1/2}\leq |u_{j}|< c\rbrace,\\
	& ds_{\hyp;s,t}^{2}=\rho_{s,t}(v_{j})\left(\frac{|dv_{j}|}{|v_{j}|\log|v_j|}\right)^{2},\,\,\,\text{on}\,\,\,B_{j}(t)=\lbrace |t_{j}|^{1/2}\leq |v_{j}|<c\rbrace.
	\end{split}
\end{displaymath}
Then, after possibly shrinking $\Omega$ in a neighborhood of $0$, there exists a constant $C>0$, independent of $s,t$, such that for all $j=1,\ldots,m$ we have
\begin{displaymath}
	\frac{1}{C}\leq \rho_{s,t}(u_{j})\leq C\,\,\,\text{on}\,\,\,A_j(t)
\end{displaymath}
and
\begin{displaymath}
		\frac{1}{C}\leq \rho_{s,t}(v_{j})\leq C\,\,\,\text{on}\,\,\,B_j(t).
\end{displaymath}
\end{proposition}
\begin{proof}[Proof of Theorem \ref{theorem:liouville}]
Let $\ff:\fX\rightarrow\Omega$ be a small stable deformation of a stable curve $X$ of genus $g\geq 2$. After possibly restricting $\Omega$, we can find meromorphic sections $\sigma$, $\tau$ of $\omega_{\fX/\Omega}$ whose divisors are relative over $\Omega$, with mutually disjoint irreducible components, disjoint from the singular points of the fibers of $\ff$. We can further assume that $\gdiv\tau$  does not meet $\cup_{j}(\overline{A_{j}(t)}\cup\overline{B_{j}(t)})$, $(s,t)\in\Omega$. We have to prove that the function
\begin{displaymath}
	(s,t)\longmapsto\log\|\langle\sigma\mid_{\fX_{s,t}},\tau\mid_{\fX_{s,t}}\rangle\|
\end{displaymath}
is continuous at $0$. Introduce the functions
\begin{displaymath}
	\begin{split}
		&F(s,t)=\int_{\fX_{s,t}}\log\|\sigma\|_{\hyp}\delta_{\gdiv\tau},\\
		&G(s,t)=\int_{\fX_{s,t}}\log\|\tau\|_{\hyp}\c1(\omega_{\fX/\Omega\hyp}),
	\end{split}
\end{displaymath}
so that $\log\|\langle\tau\mid_{\fX_{s,t}},\sigma\mid_{\fX_{s,t}}\rangle\|=F(s,t)+G(s,t)$. It suffices to show that $F$ and $G$ are continuous at $0$. The continuity of $F$ at $0$ is a consequence of the flatness of $\gdiv\sigma$ over $\Omega$ and \cite[Exp. 4.2, p. 445]{Wolpert:hyperbolic}. For the continuity of $G$ we proceed in two steps, according to the decomposition $G=G_1+G_2$, with
\begin{displaymath}
	\begin{split}
		& G_{1}(s,t)=\int_{\fX_{s,t}\setminus\cup_{j}(A_{j}(t)\cup B_{j}(t))}\log\|\tau\|_{\hyp}\c1(\omega_{\fX/\Omega\hyp}),\\
		& G_{2}(s,t)=\int_{\cup_{j}(A_{j}(t)\cup B_{j}(t))}\log\|\tau\|_{\hyp}\c1(\omega_{\fX/\Omega\hyp}).
	\end{split}	
\end{displaymath}
\textit{Step 1}. That $G_{1}(s,t)$ is continuous readily follows from \cite[Exp. 4.2, p. 445]{Wolpert:hyperbolic}, the curvature $-1$ constraint for the hyperbolic metric on $\fX_{s,t}$ and Lebesegue's dominate convergence theorem.\\
\textit{Step 2}. We treat $G_{2}(s,t)$. Observe that over the annulus $A_{j}(t)$, any differential form can be expressed in terms of the holomorphic coordinate $u_{j}$ (even for $t=0$, provided we exclude $u_{j}=0$). In the coordinate $u_j$ we have the pointwise convergence
\begin{displaymath}
	\begin{split}
 	\log(\|\tau\|_{\hyp}\mid_{A_{j}(t)})&\c1(\omega_{\fX/\Omega\hyp})\mid_{A_{j}(t)}\to\\
	&\log(\|\tau\|_{\hyp}\mid_{A_{j}(0)})\c1(\omega_{\fX/\Omega\hyp})\mid_{A_{j}(0)}\,\,\,\text{as}\,\,\,(s,t)\to 0.
	\end{split}
\end{displaymath}
Indeed, this is a consequence of \cite[Exp. 4.2, p. 445]{Wolpert:hyperbolic} and the curvature $-1$ condition for the hyperbolic metric on $\fX_{s,t}$. The corresponding fact is true for the annuli $B_{j}(t)$, as well. Now, by assumption, $|\gdiv\tau|\cap\overline{A_{j}(t)}=\emptyset$, $(s,t)\in\Omega$. From Proposition \ref{proposition:Masur} and the curvature $-1$ constraint, we derive a uniform bound
\begin{displaymath}
	\left|\log(\|\tau\|_{\hyp}\mid_{A_{j}(t)})\c1(\omega_{\fX/\Omega\hyp})\mid_{A_{j}(t)}\right|\ll\log\log|u_{j}|^{-1}\frac{|du_{j}\wedge d\overline{u}_{j}|}{|u_{j}|^{2}(\log|u_{j}|)^{2}}.
\end{displaymath}
Notice that we used that $\gdiv\tau$ is away from the singular points of the fibers of $\ff$. An analogous bound holds on $B_{j}(t)$. By Lebesgue's dominate convergence theorem we arrive to
\begin{displaymath}
	\begin{split}
	G_{2}(s,t)=\int_{\cup_{j}(A_{j}(t)\cup B_{j}(t))}&\log\|\tau\|_{\hyp}\c1(\omega_{\fX/\Omega\hyp})\\
	&\to\int_{\cup_{j}(A_{j}(0)\cup B_{j}(0))}\log\|\tau\|_{\hyp}\c1(\omega_{\fX/\Omega\hyp}).
	\end{split}
\end{displaymath}
This completes the proof of the theorem.
\end{proof}
\section{On the degeneracy of the Quillen metric}\label{section:analytic2}
\subsection{Statement of the theorem}\label{subsection:statement_theorem_selberg}
Let $(X; a_{1},\ldots, a_{n})$ be a smooth $n$-pointed stable curve of genus $g$ and $(T_{1}; b_{1})$, $\ldots$, $(T_{n}; b_{n})$ $n$ smooth 1-pointed stable curves of genus $1$, all over $\CC$. They define complex valued points $P\in\SCM_{g,n}(\CC)$, $Q_{1},\ldots, Q_{n}\in\SCM_{1,1}(\CC)$, respectively. We apply the clutching morphism $\gamma:\SCM_{g,n}\times\SCM_{1,1}^{\times n}\rightarrow\SCM_{g+n,0}$ to $(P, Q_{1},\ldots,Q_{n})$. We obtain a complex valued point $R\in\SCM_{g+n,0}(\CC)$. The curve represented by $R$ is constructed as the quotient analytic space
\begin{displaymath}
	Y=(X\sqcup T_{1}\sqcup\ldots\sqcup T_{n})/(a_{1}\sim b_{1},\ldots,a_{n}\sim b_{n}).
\end{displaymath}
Since $Y$ is compact, Chow's theorem ensures the algebraicity of $Y$. 

Construct a small stable deformation $\ff:\fY\rightarrow\Omega$ of $Y$ as described by Construction \ref{construction:deformation}. We build the family $\fg:\fZ\rightarrow D$ by restriction of $\ff$ to the locus $s_{1}=\ldots=s_{r}=0$ and $t_{1}=\ldots=t_{n}=t\in D$. The fiber $\fZ_{t}=\fg^{-1}(t)$ is non-singular for $t\neq 0$, of genus $g+n\geq 2$. Let $X^{\circ}:=X\setminus\lbrace a_{1},\ldots,a_{n}\rbrace$, $T_{j}^{\circ}=T_{j}\setminus\lbrace b_{j}\rbrace$, $j=1,\ldots,n$. Following the conventions of Section \ref{section:conventions}, we denote by $Z(\fZ_{t},s)$, $Z(X^{\circ},s)$ and $Z(T_{j}^{\circ},s)$ the Selberg zeta functions of $\fZ_{t}$, $X^{\circ}$ and $T_{j}^{\circ}$, $t\neq 0$, $j=1,\ldots,n$, respectively.
\begin{theorem}\label{theorem:Wolpert-Burger}
For $t\in D\setminus\lbrace 0\rbrace$, we have the convergence
\begin{equation}\label{equation:12}
	\begin{split}
		Z^{\prime}(\fZ_{t},1)&|t|^{-n/6}\to\\
		&\frac{1}{\pi^{n}}\left(\frac{n}{2g-2+n}+1\right)Z^{\prime}(X^{\circ},1)\prod_{j}Z^{\prime}(T_{j}^{\circ},1)\,\,\,\text{as}\,\,\,t\to 0.
	\end{split}
\end{equation}
\end{theorem}
The proof of the theorem is detailed throughout the next subsections.
\subsection{Degeneracy of the Selberg zeta function}\label{subsection:degeneracy_selberg}
We undertake the notations in Theorem \ref{theorem:Wolpert-Burger}. For every $t\in D$, $t\neq 0$, denote by $\Delta_{t}=d^{\ast}d$ the scalar hyperbolic laplacian on $\fZ_{t}$. We notice that $\Delta_{t}$ is obtained, via a fuchsian uniformization, by descent of $-y^{2}(\pd^{2}/\pd x^{2}+\pd^{2}/\pd y^{2})$ on $\HH$. It is well-known that $\Delta_{t}$ admits a unique non-negative and self-adjoint extension to the Hilbert space $L^{2}(\fZ_{t},\CC)$ \cite{Iwaniec}. For the eigenvalues of $\Delta_{t}$, counted with multiplicity, we write $\lambda_{0}(t)=0<\lambda_{1}(t)\leq\lambda_{2}(t)\leq\ldots$.
\begin{theorem}[Wolpert \textit{et al.}]\label{theorem:Wolpert2}
i. As $t\to 0$, the eigenvalues that converge to $0$ are exactly $\lambda_{1}(t),\ldots,\lambda_{n}(t)$.

ii. Let $\gamma_{1}(t),\ldots,\gamma_{n}(t)\subset\fZ_{t}$ be the simple closed geodesics that are pinched to a node as $t\to 0$. Then the holomorphic function
\begin{displaymath}
	\frac{Z(\fZ_{t},s)}{\prod_{j}Z_{l(\gamma_{j}(t))}(s)},\,\,\,\Real s>\frac{1}{2}
\end{displaymath}
uniformly converges to $Z(X^{\circ},s)\prod_{j}Z(T_{j}^{\circ},s)$ as $t\to 0$.

iii. Let $K_{1}\subset D(1,1/2)$ and $K_{2}\subset D$ be relatively compact open subsets. Then there is a uniform bound
\begin{displaymath}
	\left|\frac{Z(\fZ_{t},s)}{\prod_{j}Z_{l(\gamma_{j}(t))}(s)\prod_{j=0}^{n}(s^{2}-s+\lambda_{j}(t))}\right|\leq\beta,\,\,\,s\in K_{1},\,\,\,t\in K_{2}\setminus\lbrace 0\rbrace.
\end{displaymath}
\end{theorem}
\begin{proof}
The first assertion is established in \cite{SWY}. The second and third items are proven in \cite[Th. 35 and Th. 38]{Schulze} and \cite[proof of Conj. 1 and Conj. 2]{Wolpert:Selberg}.
\end{proof}
\begin{theorem}\label{theorem:Wolpert3}
For $t\in D\setminus\lbrace 0\rbrace$, we have the convergence
\begin{equation}\label{equation:13}
\begin{split}
\frac{1}{(2\pi)^{n}}Z^{\prime}(\fZ_{t},1)\prod_{j=1}^{n}\frac{l(\gamma_{j}(t))}{\lambda_{j}(t)}&\exp\left(\frac{\pi^{2}}{3l(\gamma_{j}(t))}\right)\to\\
	&Z^{\prime}(X^{\circ},1)\prod_{j=1}^{n}Z^{\prime}(T_{j}^{\circ},1)\,\,\,\text{as}\,\,\,t\to 0.
\end{split}
\end{equation}
\end{theorem}
\begin{proof}
For every $t\in D$, define the meromorphic function
\begin{displaymath}
	Q_{t}(s)=\frac{Z(\fZ_{t},s)}{\prod_{j}Z_{l(\gamma_{j}(t))}(s)\prod_{j=0}^{n}(s^{2}-s+\lambda_{j}(t))},\,\,\,s\in D(1,1/2).
\end{displaymath}
By Theorem \ref{theorem:Wolpert2} \textit{iii}, $Q_{t}(s)$ extends to a holomorphic function on $D(1,1/2)$. Furthermore, for every sequence $\lbrace t_{n}\rbrace_{n}\to 0$, $t_{n}\neq 0$, there exists a subsequence $\lbrace t_{n_k}\rbrace$ such that $\lbrace Q_{t_{n_k}}\rbrace_{k}$ converges to a holomorphic function $H$, uniformly over compact subsets of $D(1,1/2)$ (Montel's theorem). In particular $Q_{t_{n_k}}(1)\to H(1)$ as $k\to +\infty$. By Theorem \ref{theorem:Wolpert2} \textit{i}--\textit{ii}, we see that 
\begin{displaymath}
	H(s)=\frac{Z(X^{\circ},s)}{s^{2}-s}\prod_{j=1}^{n}\frac{Z(T_{j}^{\circ},s)}{s^{2}-s},\,\,\,s\in D(1,1/2).
\end{displaymath}
The Selberg zeta functions $Z(\fZ_{t},s)$, $Z(X^{\circ},s)$ and $Z(T_{j}^{\circ},s)$ all have a simple zero at $s=1$. The local factors $Z_{l(\gamma_{j}(t))}(s)$ are holomorphic and non-vanishing at $s=1$. Thus we compute
\begin{displaymath}
	\begin{split}
		&Q_{t}(1)=\frac{Z^{\prime}(\fZ_{t},1)}{\prod_{j=1}^{n}Z_{l(\gamma_{j}(t))}(1)\lambda_{j}(t)},\\
		&H(1)=Z^{\prime}(X^{\circ},1)\prod_{j=1}^{n}Z^{\prime}(T_{j}^{\circ},1).
	\end{split}
\end{displaymath}
Furthermore the following asymptotic estimate evaluated at $s=1$ holds \cite[Lemma 39]{Schulze}: for $\Real s>0$ we have
\begin{displaymath}
	\lim_{l\to 0}\Gamma(s)^{2}Z_{l}(s)\exp(\pi^{2}/3l)l^{2s-1}=2\pi.
\end{displaymath}
The proof follows from these considerations.
\end{proof}
\subsection{Asymptotics of small eigenvalues and proof of Theorem \ref{theorem:Wolpert-Burger}}
The notations of Section \ref{subsection:degeneracy_selberg} are still in force.
\begin{theorem}\label{theorem:Burger}
The product of the eigenvalues $\lambda_{1}(t),\ldots,\lambda_{n}(t)$ satisfies
\begin{equation}\label{equation:14}
	\lim_{t\to 0}\prod_{j=1}^{n}\frac{\lambda_{j}(t)}{l(\gamma_{j}(t))}=\frac{1}{(2\pi^{2})^{n}}\left(\frac{n}{2g-2+n}+1\right).
\end{equation}
\end{theorem}
This result actually appears as a particular instance of \cite[Th. 1.1]{Burger} (see Theorem \ref{theorem:Burger2} below), that the author learned from \cite[Sec. 4]{Ji}. We shall detail how to apply loc. cit. to derive Theorem \ref{theorem:Burger}.

The first step is to attach a \textit{weighted graph} to every fiber $\fZ_{t}$, $t\neq 0$. Let $G$ be the graph consisting of:
\begin{itemize}
	\item[--] $n+1$ vertices, $V=\lbrace v_{0},\ldots v_{n}\rbrace$;
	\item[--] $n$ edges, $E=\lbrace e_{1},\ldots e_{n}\rbrace$. The edge $e_{i}$ links the vertices $v_{0}$ and $v_{i}$.
\end{itemize}
Therefore the shape of $G$ is a $n$-edged star. For every $t\in D$, $t\neq 0$, we attach to $\fZ_{t}$ two functions $m_{t}$ and $l_{t}$:
\begin{displaymath}
	\begin{split}
		m_{t}: V&\longrightarrow\RR\\
			v_{0}&\longmapsto 2g-2+n,\\
			v_{j}&\longmapsto 1, j=1,\ldots,n
	\end{split}
\end{displaymath}
and
\begin{displaymath}
	\begin{split}
		l_{t}: E&\longmapsto\RR\\
		e_{j}&\longmapsto l(\gamma_{j}(t)).
	\end{split}
\end{displaymath}
Define $\RR[V]:=\bigoplus_{j=0}^{n}\RR v_{j}$ and $\RR[V]^{\ast}$ the vector space of linear functionals on $\RR[V]$. For every $F\in\RR[V]^{\ast}$ we put
\begin{displaymath}
	Q_{t}F=\sum_{j=1}^{n}(F(v_{j})-F(v_{0}))^{2}l_{t}(e_j).
\end{displaymath}
We also introduce the scalar product on $\RR[V]^{\ast}$
\begin{displaymath}
	\langle F_{1}, F_{2}\rangle_{t}=\sum_{j=0}^{n}m_{t}(v_{j})F_{1}(v_{j})F_{2}(v_{j}).
\end{displaymath}
There exists a unique symmetric operator $M_{t}\in\text{End}_{\RR}(\RR[V]^{\ast})$ such that
\begin{displaymath}
	\langle M_{t}F,F\rangle_{t}=Q_{t}F,\,\,\,F\in\RR[V]^{\ast}.
\end{displaymath}
The matrix of $M_{t}$ in the base $v_{0}^{\ast},\ldots,v_{n}^{\ast}$ dual to $v_{0},\ldots,v_{n}$ is
\begin{displaymath}
A(t):=\left(\begin{array}{ccccc}
	(l_{1}(t)+\ldots+l_{n}(t))/\alpha	&-l_{1}(t)/\alpha	&\cdots	&\cdots	&-l_{n}(t)/\alpha\\
	-l_{1}(t)	&l_{1}(t)	&0	&\cdots	&0\\
	\vdots	&0	&\ddots	&	&0\\
	\vdots	&\vdots	&	&\ddots	&\vdots\\
	-l_{n}(t)	&0	&0	&\cdots	&l_{n}(t)
	\end{array}\right),
\end{displaymath}
where $l_{j}(t)=l(\gamma_{j}(t))$, $\alpha=2g-2+n$. It is easily seen that $0$ is a simple eigenvalue of $A(t)$. Let $\mu_{1}(t)\leq\ldots\leq\mu_{n}(t)$ be the strictly positive eigenvalues of $M_{t}$, counted with multiplicity.
\begin{theorem}[Burger]\label{theorem:Burger2}
For every $j=1,\ldots,n$ we have
\begin{displaymath}
	\lim_{t\to 0}\frac{\lambda_{j}(t)}{\mu_{j}(t)}=\frac{1}{2\pi^{2}}.
\end{displaymath}
\end{theorem}
\begin{proof}
See \cite[Th. 1.1 and App. 2]{Burger}.
\end{proof}
Therefore, to prove Theorem \ref{theorem:Burger}, it suffices to establish the next proposition.
\begin{proposition}\label{proposition:Burger}
For the product of the eigenvalues $\mu_{1}(t),\ldots,\mu_{n}(t)$ we have
\begin{displaymath}
	\lim_{t\to 0}\prod_{j=1}^{n}\frac{\mu_{j}(t)}{l(\gamma_{j}(t))}=\frac{n}{2g-2+n}+1.
\end{displaymath}
\end{proposition}
\begin{proof}
Introduce the function $l(t)=2\pi^{2}/\log|t|^{-1}$, $t\in D$, $t\neq 0$. By \cite[Ex. 4.3, p. 446]{Wolpert:hyperbolic} there is an estimate
\begin{equation}\label{equation:15}
	l(\gamma_{j}(t))=l(t)+O\left(\frac{1}{(\log|t|)^{4}}\right).
\end{equation}
Define the matrix
\begin{displaymath}
B(t):=\left(\begin{array}{ccccc}
	nl(t)/\alpha	&-l(t)/\alpha	&\cdots	&\cdots	&-l(t)/\alpha\\
	-l(t)	&l(t)	&0	&\cdots	&0\\
	\vdots	&0	&\ddots	&	&0\\
	\vdots	&\vdots	&	&\ddots	&\vdots\\
	-l(t)	&0	&0	&\cdots	&l(t)
	\end{array}\right).
\end{displaymath}
Let $\nu_{0}(t)=0<\nu_{1}(t)\leq\ldots\leq\nu_{n}(t)$ be the eigenvalues of $B(t)$. They are easily computed to be
\begin{displaymath}
	\nu_{1}(t)=\ldots=\nu_{n-1}(t)=l(t)<\nu_{n}(t)=\left(\frac{n}{2g-2+n}+1\right)l(t).
\end{displaymath}
By comparison of the characteristic polynomials of $A(t)$, $B(t)$ and the estimate (\ref{equation:15}) we arrive to
\begin{displaymath}
	\lim_{t\to 0}\prod_{j=1}^{n}\frac{\mu_{j}(t)}{l(\gamma_{j}(t))}=\lim_{t\to 0}\prod_{j=1}^{n}\frac{\nu_{j}(t)}{l(t)}=\frac{n}{2g-2+n}+1.
\end{displaymath}
This completes the proof of the proposition.
\end{proof}
\begin{proof}[Proof of Theorem \ref{theorem:Burger}]
The theorem is a straightforward consequence of Theorem \ref{theorem:Burger2} and Proposition \ref{proposition:Burger}.
\end{proof}
We are now in position to prove Theorem \ref{theorem:Wolpert-Burger}.
\begin{proof}[Proof of Theorem \ref{theorem:Wolpert-Burger}]
The limit property (\ref{equation:12}) is a conjunction of (\ref{equation:13}) (Theorem \ref{theorem:Wolpert3}), (\ref{equation:14}) (Theorem \ref{theorem:Burger}) and Wolpert's estimate (\ref{equation:15}) for the length $l(\gamma_{j}(t))$.
\end{proof}
\subsection{Consequences: degeneracy of the Quillen metric}\label{subsection:degeneracy_Quillen}
We apply Theorem \ref{theorem:Wolpert-Burger} to study the behavior of the Quillen metric $\|\cdot\|_{Q}$ on $\lambda_{g+n,0}$, near $\pd\SM_{g,0}^{\an}$.

We maintain the notations of Section \ref{subsection:statement_theorem_selberg}. For the clutching morphism $\gamma:\SCM_{g,n}\times\SCM_{1,1}^{\times n}\rightarrow\SCM_{g+n,0}$, we have the isomorphism (\ref{equation:3bis})
\begin{displaymath}
	\Phi:\gamma^{\ast}\lambda_{g+n,0}\overset{\sim}{\longrightarrow}\lambda_{g,n}\boxtimes\lambda_{1,1}^{\boxtimes n}.
\end{displaymath}
 At the point $N=(P,Q_{1},\ldots,Q_{n})$, $\Phi$ induces a natural isomorphism
 \begin{displaymath}
 	\Phi_{N}:\lambda(\omega_{\fZ_{0}})\overset{\sim}{\longrightarrow}\lambda(\omega_{X})\otimes\bigotimes_{j=1}^{n}\lambda(\omega_{T_j}).
 \end{displaymath}
 Observe that $\lambda(\omega_{\fZ_{0}})=\lambda(\omega_{Y})$ is the stalk at $0\in D$ of $\lambda(\omega_{\fZ/D})$.
 \begin{proposition}\label{proposition:L2_metric}
 i. The $L^{2}$ metric $\|\cdot\|_{L^2}$ on $\lambda(\omega_{\fZ/D})\mid_{D\setminus\lbrace 0\rbrace}$ extends continuously to $\lambda(\omega_{\fZ/D})$.\\
 ii. Let $\mu=n/(2g-2+n)+1$. Then $\Phi_{N}$ induces an isometry
 \begin{displaymath}
 	\begin{split}
	\overline{\Phi}_{N}:(\lambda(\omega_{\fZ/D}),&\|\cdot\|_{L^2})\mid_{0}\overset{\sim}{\longrightarrow}\\
	&(\lambda(\omega_{X}),\|\cdot\|_{L^2})
	\otimes\bigotimes_{j=1}^{n}(\lambda(\omega_{T_j}),\|\cdot\|_{L^{2}})\otimes\OO(\gamma^{1/2}).
	\end{split}
 \end{displaymath}
 \end{proposition}
 \begin{proof}[Proof (see also \cite{BB}, Prop. 13.5, case (II), pp. 96--98)]
 We shall prove the two assertions simultaneously. Write $D^{\ast}=D\setminus\lbrace 0\rbrace$. First of all we have isometries
 \begin{displaymath}
 	\begin{split}
		&(\lambda(\omega_{\fZ/D})\mid_{D^{\ast}},\|\cdot\|_{L^2})\overset{\sim}{\rightarrow}(\det R^{0}\fg_{\ast}\omega_{\fZ/D}\mid_{D^{\ast}},V_{g+n,0}^{1/2}\cdot\|\cdot\|_{0}),\\
		&(\lambda(\omega_{X}),\|\cdot\|_{L^2})\overset{\sim}{\rightarrow}(\det H^{0}(X,\omega_X),V_{g,n}^{1/2}\cdot\|\cdot\|_{0}),\\
		&(\lambda(\omega_{T_j}),\|\cdot\|_{L^2})\overset{\sim}{\rightarrow}(\det H^{0}(T_{j},\omega_{T_j}), V_{1,1}^{1/2}\cdot\|\cdot\|_{0}),
	\end{split}
 \end{displaymath}
 where $\|\cdot\|_{0}$ is the $L^{2}$ metric on the line $\det R^{0}\fg_{\ast}\omega_{\fZ/D}\mid_{D^{\ast}}$ (respectively $\det H^{0}(X,\omega_{X})$, $\det H^{0}(T_{j},\omega_{T_j})$) and $V_{g+n,0}=2g+2n-2$, $V_{g,n}=2g-2+n$, $V_{1,1}=1$. Since $\mu=V_{g+n,0}/V_{g,n}$, it suffices to prove that the $L^{2}$ metric on $\det R^{0}\fg_{\ast}\omega_{\fZ/D}\mid_{D^{\ast}}$ extends continuously at $0$ and $\Phi_{N}$ induces an isometry
 \begin{displaymath}
 	\begin{split}
	(\det R^{0}\fg_{\ast}\omega_{\fZ/D},&\|\cdot\|_{0})\mid_{0}\overset{\sim}{\longrightarrow}\\
	&(\det H^{0}(X,\omega_{X}),\|\cdot\|_{0})
	\otimes\bigotimes_{j=1}^{n}(\det H^{0}(T_{j},\omega_{T_{j}}),\|\cdot\|_{0}).
	\end{split}
 \end{displaymath}
Let $\alpha_{1},\ldots,\alpha_{g}$ be a basis of $H^{0}(X,\omega_{X})$ and $\beta_{j}$ a basis of $H^{0}(T_{j},\omega_{T_j})$, $j=1,\ldots,n$. The differential forms $\alpha_{i}$, $\beta_{j}$ satisfy the residue conditions $\Res_{p_k}\alpha_{i}=0$, $\Res_{q_{j}}\beta_{j}=0$. Therefore there exist global sections $\widetilde{\alpha}_{i}$ and $\widetilde{\beta}_{j}$ of $\omega_{Y}$ extending $\alpha_{i}$, $\beta_{j}$ by $0$, respectively. The sections $\widetilde{\alpha}_{i}$, $i=1,\ldots,g$, $\widetilde{\beta}_{j}$, $j=1,\ldots,n$ form a basis of $H^{0}(Y,\omega_{Y})$. Besides, $R^{0}\fg_{\ast}\omega_{\fZ/D}$ is locally free of rank $g+n$. After possibly shrinking $D$, we can find a frame $\widetilde{\alpha}_{i}(t)$, $i=1,\ldots,g$, $\widetilde{\beta}_{j}(t)$, $j=1,\ldots,n$ of $R^{0}\fg_{\ast}\omega_{\fZ/D}$ over $D$, such that $\widetilde{\alpha}_{i}(0)=\widetilde{\alpha}_{i}$ and $\widetilde{\beta}_{j}(0)=\widetilde{\beta}_{j}$. Write $\lbrace\theta_{i}(t)\rbrace_{i=1}^{g+n}$ for the whole ordered set $\lbrace\widetilde{\alpha}_{i}(t)\rbrace_{i=1}^{g}\cup\lbrace\widetilde{\beta}_{j}(t)\rbrace_{j=1}^{n}$. An easy local computation shows that the function
 \begin{displaymath}
 	t\mapsto\det\left(\frac{i}{2\pi}\int_{\fZ_{t}}\theta_{j}\wedge\overline{\theta_{k}}\right)_{1\leq j\leq k\leq g+n},\,\,\,t\neq 0
 \end{displaymath}
 extends continuously to $0$, with value
 \begin{displaymath}
 	\det\left(\frac{i}{2\pi}\int_{X}\alpha_{j}\wedge\overline{\alpha_{k}}\right)_{1\leq j,k\leq g}
	\prod_{j=1}^{n}\frac{i}{2\pi}\int_{T_{j}}\beta_{j}\wedge\overline{\beta_{j}}.
 \end{displaymath}
 Notice that
 \begin{displaymath}
	\Phi_{N}:(\theta_{1}\wedge\ldots\wedge\theta_{g+n})(0)\mapsto\pm(\alpha_{1}\wedge\ldots\wedge\alpha_{g})\otimes\beta_{1}\otimes\ldots\otimes\beta_{n}.
 \end{displaymath}
The proposition follows from these observations.
 \end{proof}
 Attached to the one-parameter deformation $\fg:\fZ\rightarrow D$, there is a classifying map
 \begin{displaymath}
 	\C(\fg):D\longrightarrow\SCM_{g+n,0}^{\an}.
\end{displaymath}
The line bundle $\lambda(\omega_{\fZ/D})=\C(\fg)^{\ast}\lambda_{g+n,0}^{\an}$ is endowed with the Quillen metric $\|\cdot\|_{Q}$.
\begin{corollary}\label{corollary:L2_metric}
Let $\widetilde{\alpha}_{1},\ldots,\widetilde{\alpha}_{g}$, $\widetilde{\beta}_{1},\ldots,\widetilde{\beta}_{n}$ be a frame of the sheaf $R^{0}\fg_{\ast}\omega_{\fZ/D}$ over $D$, as in the proof of Proposition \ref{proposition:L2_metric}. Then we have
\begin{displaymath}
	\begin{split}
	\lim_{t\to 0}\|\widetilde{\alpha}_{1}\wedge\ldots\wedge\widetilde{\alpha}_{g}
	\wedge&\widetilde{\beta}_{1}\wedge\ldots\wedge\widetilde{\beta}_{n}\|_{Q,t}|t|^{n/12}\\
	&=\|\alpha_{1}\wedge\ldots\wedge\alpha_{g}\|_{Q}\cdot
	\ldots\cdot\|\beta_{1}\|_{Q}\cdot\ldots\cdot\|\beta_{n}\|_{Q}.
	\end{split}
\end{displaymath}
\end{corollary}
\begin{proof}
This is easily derived from the very definition of $\|\cdot\|_{Q}$ (see Section \ref{section:conventions}), Theorem \ref{theorem:Wolpert-Burger}, Proposition \ref{proposition:L2_metric} and the relation (\ref{equation:E}).
\end{proof}
\section{Metrized Mumford isomorphism on $\SM_{g,n}$}\label{section:Mumford_isometry}
In this section we establish the following metrized version of the Mumford isomorphism (Theorem \ref{theorem:Mumford_isomorphism}).
\begin{theorem}\label{theorem:Mumford_isometry}
Let $\DD_{g,n}^{\circ}$ be the restriction to $\SM_{g,n}$ of the Mumford isomorphism $\DD_{g,n}$. Then $\DD_{g,n}^{\circ}$ induces an isometry
\begin{displaymath}
	\overline{\DD}_{g,n}^{\circ}:\lambda_{g,n;Q}^{\otimes 12}\otimes\psi_{g,n;W}\overset{\sim}{\longrightarrow}\overline{\kappa}_{g,n}\otimes\OO(C(g,n))\,\,\,\text{on}\,\,\,
	\SM_{g,n}.
\end{displaymath}
\end{theorem}
\begin{remark}
For the sake of simplicity, in the theorem we wrote $\lambda_{g,n}$, $\psi_{g,n}$, etc. instead of $\lambda_{g,n}\mid_{\SM_{g,n}}$, $\psi_{g,n}\mid_{\SM_{g,n}}$ etc. respectively.
\end{remark}
To lighten the forthcoming arguments, it is worth introducing some notations.
\begin{notation}\label{notation:Mumford_isometry}
We define the line bundles
\begin{displaymath}
	\begin{split}
		&L_{g,n}:=\lambda_{g,n}^{\otimes 12}\otimes\psi_{g,n}\otimes\delta_{g,n}^{-1}.\\
		&\fL_{g,n}:=(\lambda_{g,n}^{\otimes 12}\otimes\psi_{g,n}\otimes\delta_{g,n}^{-1})\boxtimes(\lambda_{1,1}^{\otimes 12}\otimes\psi_{1,1}\otimes\delta_{1,1}^{-1})^{\boxtimes n}
	\end{split}
\end{displaymath}
and the isomorphism
\begin{displaymath}
	\fD_{g,n}:=\DD_{g,n}\boxtimes\DD_{1,1}^{\boxtimes n}:\fL_{g,n}\overset{\sim}{\longrightarrow}\kappa_{g,n}\boxtimes\kappa_{1,1}^{\boxtimes n}.
\end{displaymath}
We let $\|\cdot\|_{g,n}$ be the continuous hermitian metric on $L_{g,n}$ such that the Mumford isomorphism $\DD_{g,n}$ becomes an isometry
\begin{displaymath}
	\overline{\DD}_{g,n}:\overline{L}_{g,n}:=(L_{g,n},\|\cdot\|_{g,n})\overset{\sim}{\longrightarrow}\overline{\kappa}_{g,n}\otimes\OO(C(g,n)).
\end{displaymath}
Finally, we write $\overline{\fL}_{g,n}=\overline{L}_{g,n}\boxtimes\overline{L}_{1,1}^{\boxtimes n}$ and $\overline{\fD}_{g,n}=\overline{\DD}_{g,n}\boxtimes\overline{\DD}_{1,1}^{\boxtimes n}$.
\end{notation}
The first observations towards the proof of Theorem \ref{theorem:Mumford_isometry} is summarized in the next proposition.
\begin{proposition}\label{proposition:key}
Let $g\geq 2$ be an integer. Then:

i. Theorem \ref{theorem:Mumford_isometry} holds true for $(g,n)=(g,0)$;

ii. endow $\delta_{g,0}^{-1}=\OO(-\pd\SM_{g,0})$ with the trivial singular metric coming from the absolute value; write $\overline{\delta}_{g,0}^{-1}$ for the resulting hermitian line bundle. Then $\lambda_{g,0;Q}^{\otimes 12}$ extends to a continuous hermitian line bundle $\lambda_{g,0;Q}^{\otimes 12}\otimes\overline{\delta}_{g,0}^{-1}$ on $\SCM_{g,0}$. Moreover, $\overline{\DD}^{\circ}_{g,0}$ extends to an isometry
\begin{displaymath}
	\lambda_{g,0;Q}^{\otimes 12}\otimes\overline{\delta}_{g,0}^{-1}\overset{\sim}{\longrightarrow}
	\overline{\kappa}_{g,0}\otimes\OO(C(g,0)).
\end{displaymath}
\end{proposition}
\begin{proof}
In the present form, the first item is a theorem of Deligne \cite[Th. 11.4]{Deligne} and Gillet-Soul\'e \cite{GS}. Indeed, it is enough to point out the following facts:
\begin{itemize}
	\item[--] Deligne's functorial isomorphism coincides with $\DD^{\circ}_{g,0}$, up to a sign. This is justified by $H^{0}(\SM_{g,0},\Gm)=\lbrace\pm 1\rbrace$ \cite[Lemma 2.2.3]{MB};
	\item[--] the normalization of the Quillen metric on $\lambda(\omega_X)$ of Deligne and Gillet-Soul\'e in loc. cit. coincides with ours (see Section \ref{section:conventions}). Let $X$ be a compact Riemann surface of genus $g$, with hyperbolic metric of curvature $-1$, $ds_{\hyp}^{2}$. Let $h$ be the hermitian metric induced by $ds_{\hyp}^{2}$ on $T_{X}$. Denote by $\Delta_{\cpd}=\cpd^{\ast}\cpd$ the associated $\cpd$-laplacian acting on functions, and $\Delta_{d}=d^{\ast}d$ the hyperbolic scalar laplacian. Recall the K\"ahler identity \cite[Ch. 5]{Voisin}
\begin{equation}\label{equation:16}
	\Delta_{\cpd}=\frac{1}{2}\Delta_{d}.
\end{equation}
Deligne and Gillet-Soul\'e work with the Quillen metric
\begin{displaymath}
	\|\cdot\|=(\detp\Delta_{\cpd})^{-1/2}\|\cdot\|_{L^{2}}.\footnote{This definition agrees with the one by Deligne and Gillet-Soul\'e due to the remark preceding \cite[Prop. 2.7, p. 159]{RS}.}
\end{displaymath}
We now check the relation
\begin{equation}\label{equation:16bis}
 	\detp\Delta_{\cpd}=E(g,0)Z^{\prime}(X,1).
\end{equation}
First of all, since $X$ has genus $g$ and by (\ref{equation:16}), we compute
\begin{equation}\label{equation:17}
	\detp\Delta_{\cpd}=2^{(g+2)/3}\detp\Delta_{d}
\end{equation}
(see \cite[Eq. 13.2, p. 88]{BB}). By a theorem of d'Hoker-Phong \cite{dHP} and Sarnak \cite[Cor. 1]{Sarnak}, we have the expression
\begin{equation}\label{equation:18}
	\detp\Delta_{d}=Z^{\prime}(X,1)\exp((2g-2)(2\zeta^{\prime}(-1)-1/4+1/2\log2\pi)).
\end{equation}
Equations (\ref{equation:17})--(\ref{equation:18}) together already imply (\ref{equation:16bis}), and hence the equality $\|\cdot\|=\|\cdot\|_{Q}$.
\end{itemize}
The second assertion \textit{ii} is derived by combination of the Mumford isomorphism on $\SCM_{g,0}$ (Theorem \ref{theorem:Mumford_isomorphism}), the continuity of the Liouville metric (Theorem \ref{theorem:liouville}) and the first point \textit{i}.
\end{proof}
\begin{corollary}\label{corollary:key}
i. For every integer $g\geq 2$, the continuous hermitian line bundle $\overline{L}_{g,0}=\overline{\fL}_{g,0}$ satisfies the following factorization formula:
\begin{displaymath}
	\overline{\fL}_{g,0}=\lambda_{g,0;Q}^{\otimes 12}\otimes\overline{\delta}_{g,0}^{-1}.
\end{displaymath}

ii. There is a diagram of isometries of continuos hermitian line bundles on $\SCM_{g,n}$
\begin{equation}\label{equation:19}
	\xymatrix{
		\gamma^{\ast}\overline{\fL}_{g+n,0}\ar[r]^{\overset{\hspace{-1cm}\gamma^{\ast}\overline{\DD}_{g+n,0}}{\hspace{-1cm}\sim}}\ar[d]_{\wr}
		&\gamma^{\ast}(\overline{\kappa}_{g+n,0}\otimes\OO(C(g+n,0)))\ar[d]^{\wr}\\
		\overline{\fL}_{g,n}\ar[r]_{\underset{\hspace{-3cm}\overline{\fD}_{g,n}}{\hspace{-3cm}\sim}}
		&(\overline{\kappa}_{g,n}\otimes\OO(C(g,n)))\boxtimes(\overline{\kappa}_{1,1}\otimes O(C(1,1)))^{\boxtimes n},
	}
\end{equation}
commutative up to a sign. The isomorphisms underlying the vertical arrows are induced by (\ref{equation:3bis})--(\ref{equation:6bis}) (Corollary \ref{corollary:clutching}).
\end{corollary}
\begin{proof}
The first claim is a reformulation of Proposition \ref{proposition:key} \textit{ii}. For the second assertion, we first observe that if we forget the hermitian structures, then (\ref{equation:19}) is a consequence of Corollary \ref{corollary:irreducible_fibers}, Corollary \ref{corollary:clutching} and Corollary \ref{corollary:Mumford_isomorphism}. The existence of the whole diagram (\ref{equation:19}) is a conjunction of the very definition of $\overline{\fL}_{g,n}$ and $\overline{\fD}_{g,n}$ (Notation \ref{notation:Mumford_isometry}), relation (\ref{equation:C}), Lemma \ref{lemma:clutching_liouville} and Proposition \ref{proposition:key} \textit{ii}.
\end{proof}
Theorem \ref{theorem:Mumford_isometry} is actually equivalent to the next apparently weaker statement.
\begin{proposition}\label{proposition:Mumford_isometry}
There is a factorization of hermitian line bundles on $\SM_{g,n}\times\SM_{1,1}^{\times n}$
\begin{equation}\label{equation:20}
	\overline{\fL}_{g,n}\mid_{\SM_{g,n}\times\SM_{1,1}^{\times n}}=(\lambda_{g,n;Q}^{\otimes 12}\otimes\psi_{g,n;W})\boxtimes(\lambda_{1,1;Q}^{\otimes 12}\otimes\psi_{1,1;Q})^{\boxtimes n}.
\end{equation}
\end{proposition}
\begin{proof}
We first observe that proving (\ref{equation:20}) is tantamount to proving the factorization formula
\begin{equation}\label{equation:21}
	\overline{\fL}_{g,n}^{\an}\mid_{\SM_{g,n}^{\an}\times\SM_{1,1}^{\an\times n}}=(\lambda_{g,n;Q}^{\an\otimes 12}\otimes\psi_{g,n;W}^{\an})\boxtimes
	(\lambda_{1,1;Q}^{\an\otimes 12}\otimes\psi_{1,1;Q}^{\an})^{\boxtimes n}.
\end{equation}
Since we aim to show the factorization of the underlying hermitian structures, we shall establish (\ref{equation:21}) pointwise. 

By Corollary \ref{corollary:key} there is an isometry of continuous hermitian line bundles on $\SCM_{g,n}\times\SCM_{1,1}^{\times n}$
\begin{displaymath}
	\overline{\Psi}:\gamma^{\ast}\overline{\fL}_{g+n,0}\overset{\sim}{\longrightarrow}\overline{\fL}_{g,n}.
\end{displaymath}
The isomorphism of line bundles underlying $\overline{\Psi}$ is build up with (\ref{equation:3bis})--(\ref{equation:4bis}). Fix a complex valued point $N=(P,Q_{1},\ldots,Q_{n})\in\SM_{g,n}(\CC)\times\SM_{1,1}(\CC)^{\times n}$; let $R$ be its image in $\SCM_{g+n,0}(\CC)$ by the clutching morphism $\gamma$. At the point $N$, $\overline{\Psi}$ induces an isometry of complex hermitian lines
\begin{equation}\label{equation:21bis}
	\overline{\Psi}_{N}:R^{\ast}\overline{\fL}_{g+n,0}=N^{\ast}\gamma^{\ast}\overline{\fL}_{g+n,0}\overset{\sim}{\longrightarrow}N^{\ast}\overline{\fL}_{g,n}.
\end{equation}
We focus on $R^{\ast}\overline{\fL}_{g+n,0}$. Let $(X;a_{1},\ldots,a_{n})$, $(T_{1};b_{1}),\ldots, (T_{n},b_{n})$ be the pointed stable curves corresponding to $P$, $Q_{1},\ldots,Q_{n}$, respectively. Let $Y$ be the curve represented by $R$. By means of Construction \ref{construction:deformation} we obtain a small stable deformation $\ff:\fY\rightarrow\Omega$ of $Y$. Attached to $\ff$ there is a classifying morphism
\begin{displaymath}
	\C(\ff):\Omega\longrightarrow\SCM_{g+n,0}^{\an}.
\end{displaymath}
We agree in denoting the pull-backs of $\lambda_{g+n,0}^{\an}$, $\delta_{g+n,0}^{\an}$, $\fL_{g+n,0}^{\an}$, etc. to $\Omega$ by $\lambda(\ff)$, $\delta(\ff)$, $\fL(\ff)$, etc. respectively. We also set $R(\ff):=0\in\Omega$. We have to study the hermitian complex line $R(\ff)^{\ast}\overline{\fL}(\ff)$. With the notations of Construction \ref{construction:deformation}, the equation of the divisor of the singular fibers of $\ff$ is $t_{1}\cdot\ldots\cdot t_{n}=0$. After possibly shrinking $\Omega$, the holomorphic section $e:=t_{1}\cdot\ldots\cdot t_{n}$ is a frame of $\delta(\ff)^{-1}$. We introduce two auxiliary metrics:
\begin{itemize}
	\item[--] a modified Quillen metric $\|\cdot\|_{Q}^{\prime}$ on $\lambda(\ff)$:
		\begin{displaymath}
			 \|\cdot\|_{Q;(s,t)}^{\prime}=\|\cdot\|_{Q}^{\prime}|t_{1}\cdot\ldots\cdot t_{n}|^{1/12}\,\,\,\text{at the point}\,\,\,(s,t)\in\Omega;
		\end{displaymath}
	\item[--] the smooth metric $\|\cdot\|^{\prime}$ on $\delta(\ff)^{-1}$ defined by the rule
		\begin{displaymath}
			\|e\|^{\prime}=1.
		\end{displaymath}
\end{itemize}
With these choices, there is an obvious equality
\begin{equation}\label{equation:22}
	\overline{\fL}(\ff)=\lambda(\ff)_{Q}^{\otimes 12}\otimes\overline{\delta(\ff)}^{-1}=\lambda(\ff)^{\prime\otimes 12}_{Q}\otimes\overline{\delta(\ff)}^{\prime -1}.
\end{equation}
We point out two features concerning $\|\cdot\|_{Q}^{\prime}$ and $\|\cdot\|^{\prime}$:
\begin{itemize}
	\item[--] by the continuity of the metric of $\overline{\fL}(\ff)$, the smoothness of $\|\cdot\|^{\prime}$ and equality (\ref{equation:22}), we see that the metric $\|\cdot\|_{Q}^{\prime}$ is actually continuous;
	\item[--] the pull-back of the isomorphism (\ref{equation:4bis}) by $N$ yields
		\begin{displaymath}
			\begin{split}
				\Theta_{N}:R^{\ast}\delta_{g,n}^{-1}=R(\ff)^{\ast}\delta(\ff)^{-1}&\overset{\sim}{\longrightarrow} P^{\ast}\psi_{g,n}\otimes\bigotimes_{j=1}^{n} Q_{j}^{\ast}\psi_{1,1}\\
				R(\ff)^{\ast}e&\longmapsto \pm(\otimes_{j=1}^{n} du_{j})\otimes(\otimes_{j=1}^{n} dv_{j}),
			\end{split}
		\end{displaymath}
		where $u_{j}$, $v_{j}$ are the \textit{rs} coordinates at $a_{j}$, $b_{j}$, respectively, used to construct $\ff:\fY\rightarrow\Omega$. This is so because the degeneration of the family $\ff$ in a neighborhood of the node $a_{j}\sim b_{j}$ is modeled by $u_{j}v_{j}=t_{j}$ (see \cite[Sec. 4]{Edixhoven} for a detailed proof). Therefore, if we endow $\delta(\ff)^{-1}$ with the metric $\|\cdot\|^{\prime}$, then $\Theta_{N}$ becomes an isometry
		\begin{equation}\label{equation:22bis}
			\overline{\Theta}_{N}:R(\ff)^{\ast}\overline{\delta(\ff)}^{\prime -1}\overset{\sim}{\longrightarrow} P^{\ast}\psi_{g,n;W}\otimes\bigotimes_{j=1}^{n} Q_{j}^{\ast}\psi_{1,1;W}.
		\end{equation}
		Indeed, it suffices to recall that the Wolpert metric assigns the value $1$ to $du_{j}$ and $dv_{j}$ (see Definition \ref{definition:Wolpert_metric}).
\end{itemize}
We claim that  the isomorphism (\ref{equation:3bis}) induces an isometry
\begin{equation}\label{equation:23}
	\overline{\Phi}_{N}:R(\ff)^{\ast}\lambda(\ff)^{\prime}_{Q}\overset{\sim}{\longrightarrow}P^{\ast}\lambda_{g,n;Q}\otimes\bigotimes_{j=1}^{n}Q_{j}^{\ast}\lambda_{1,1;Q}.
\end{equation}
Let $\fg:\fZ\rightarrow D$ be the restriction of $\ff$ to the locus $s_{1}=\ldots=s_{r}=0$, $t_{1}=\ldots=t_{n}=t\in D$, for some small disc $D\subset\CC$ centered at $R(\fg):=0$. Let $\C(\fg)$ be the associated classifying map. By the very definition of $\fg$ there is a commutative diagram
\begin{equation}\label{equation:24}
	\xymatrix{
		&R(\fg)=0\,\ar[dd]\ar@{^{(}->}[r]	& D\ar@{^{(}->}[dd]_{\iota}\ar[rd]^{\C(\fg)}	&\\
		&	&	&\SCM_{g+n,0}^{\an}.\\
		&R(\ff)=0\,\ar@{^{(}->}[r]	&\Omega\ar[ru]_{\C(\ff)}		&
	}
\end{equation}
From (\ref{equation:24}) we derive the equality of line bundles
\begin{displaymath}
 	\lambda(\fg)=\C(\fg)^{\ast}\lambda_{g+n,0}^{\an}=\iota^{\ast}\lambda(\ff).
\end{displaymath}
Therefore, the pull-back of $\|\cdot\|_{Q}^{\prime}$ by $\iota$ is a continuous hermitian structure on $\lambda(\fg)$. Write $\lambda(\fg)_{Q}^{\prime}$ for the resulting hermitian line bundle. The claim (\ref{equation:23}) is equivalent to asserting an isometry induced by (\ref{equation:3bis})
\begin{equation}\label{equation:25bis}
	\overline{\Phi}_{N}:R(\fg)^{\ast}\lambda(\fg)^{\prime}_{Q}\overset{\sim}{\longrightarrow} P^{\ast}\lambda_{g,n;Q}\otimes\bigotimes_{j=1}^{n}Q_{j}^{\ast}\lambda_{1,1;Q}.
\end{equation}
The validity of (\ref{equation:25bis}) has been established in \textsection \ref{subsection:degeneracy_Quillen}, Corollary \ref{corollary:L2_metric}. We come up to the conclusion by (\ref{equation:21bis})--(\ref{equation:23}).
\end{proof}
\begin{proof}[Proof of Theorem \ref{theorem:Mumford_isometry}]
By the very definition of $\overline{L}_{g,n}$ (Notation \ref{notation:Mumford_isometry}) it suffices to show a decomposition of hermitian line bundles
\begin{equation}\label{equation:25}\tag{$\text{DEC}_{g,n}$}
	\overline{L}_{g,n}\mid_{\SM_{g,n}}=\lambda_{g,n;Q}^{\otimes 12}\otimes\psi_{g,n;W}.
\end{equation}
As a first step, we treat the case $g=n=1$. Proposition \ref{proposition:Mumford_isometry} yields
\begin{equation}\label{equation:26}
	\overline{\fL}_{1,1}\mid_{\SM_{1,1}^{\times 2}}=(\lambda_{1,1;Q}^{\otimes 12}\otimes\psi_{1,1;W})^{\boxtimes 2}.
\end{equation}
Recall that $\overline{\fL}_{1,1}\mid_{\SM_{1,1}^{\times 2}}=\overline{L}_{1,1}^{\boxtimes 2}$. Therefore the pull-back of (\ref{equation:26}) by the diagonal morphism $\Delta:\SM_{1,1}\rightarrow\SM_{1,1}\times\SM_{1,1}$ leads to
\begin{equation}\label{equation:27}
	\overline{L}_{1,1}\mid_{\SM_{1,1}}^{\otimes 2}=(\lambda_{1,1;Q}^{\otimes 12}\otimes\psi_{1,1;W})^{\otimes 2}.
\end{equation}
Because $L_{1,1}\mid_{\SM_{1,1}}=\lambda_{1,1}^{\otimes 12}\otimes\psi_{1,1}$, the formula ($\text{DEC}_{1,1}$) follows from (\ref{equation:27}). To establish (\ref{equation:25}) for general $g,n$, we tensor equation (\ref{equation:20}) in Proposition \ref{proposition:Mumford_isometry} by the identity $(pr_{2}^{\ast}(\text{DEC}_{1,1})^{\boxtimes n})^{\otimes -1}$ ($pr_{2}$ is the projection $\SM_{g,n}\times\SM_{1,1}^{\times n}\rightarrow\SM_{1,1}^{\times n}$). We obtain
\begin{equation}\label{equation:28}
	pr_{1}^{\ast}\overline{L}_{g,n}\mid_{\SM_{g,n}}=pr_{1}^{\ast}(\lambda_{g,n;Q}^{\otimes 12}\otimes\psi_{g,n;W}).
\end{equation}
By definition of $L_{g,n}$, $L_{g,n}\mid_{\SM_{g,n}}=\lambda_{g,n}^{\otimes 12}\otimes\psi_{g,n}$. Consequently, (\ref{equation:28}) already implies the coincidence of the metric on $\overline{L}_{g,n}\mid_{\SM_{g,n}}$ and the metric on $\lambda_{g,n;Q}^{\otimes 12}\otimes\psi_{g,n;W}$, hence ($\text{DEC}_{g,n}$). This completes the proof of the theorem.
\end{proof}
\begin{remark}\label{remark:psi_cont}
The Wolpert metric on $\psi_{g,n}$ naturally extends to $\SCM_{g,n}$. For the clutching morphism $\beta:\SCM_{g_{1},n_{1}+1}\times\SCM_{g_{2},n_{2}+1}\rightarrow\SCM_{g_{1}+g_{2},n_{1}+n_{2}}$ we have an isometry of hermitian line bundles
\begin{equation}\label{equation:28bis_1}
	\beta^{\ast}\psi_{g_{1}+g_{2},n_{1}+n_{2};W}\overset{\sim}{\longrightarrow}\psi_{g_{1},n_{1}+1;W}
	\boxtimes\psi_{g_{2},n_{2}+1;W}.
\end{equation}
In view of the results in \cite[Chap. 4]{GFM:thesis} --specially Section 4.3 of loc. cit.--, we expect that $\psi_{g,n;W}$ is a continuous hermitian line bundle on $\SCM_{g,n}$. In this case, the isometry $\overline{\DD}_{g,n}^{\circ}$ of Theorem \ref{theorem:Mumford_isometry} extends to an isometry of continuous hermitian line bundles
\begin{displaymath}
	\overline{\DD}_{g,n}:\lambda_{g,n;Q}^{\otimes 12}\otimes\overline{\delta}_{g,n}^{-1}\otimes
	\psi_{g,n;W}\overset{\sim}{\longrightarrow}\overline{\kappa}_{g,n}\otimes\OO(C(g,n))\,\,\,\text{on}\,\,\,
	\SCM_{g,n},
\end{displaymath}
where $\overline{\delta}^{-1}_{g,n}$ is equipped with the trivial singular metric coming from the absolute value. A parallel argument as for the proof of Theorem \ref{theorem:Mumford_isometry} and Proposition \ref{proposition:Mumford_isometry} --combining Lemma \ref{lemma:clutching_liouville}, equation (\ref{equation:28bis_1}), \cite[Th. 35 and Th. 38]{Schulze} and \cite[Th. 1.1]{Burger}-- will then lead to a factorization
\begin{equation}\label{equation:28bis_2}
	\beta^{\ast}\overline{\DD}_{g_{1}+g_{2},n_{1}+n_{2}}=\overline{\DD}_{g_{1},n_{1}+1}\boxtimes
	\overline{\DD}_{g_{2},n_{2}+1}.
\end{equation}
An analogous compatibility formula is expected for the clutching morphism $\alpha:\SCM_{g-1,n+2}\rightarrow\SCM_{g,n}$ of \cite[Def. 3.8 and Th. 4.2]{Knudsen}:
\begin{equation}\label{equation:28bis_3}
	\alpha^{\ast}\overline{\DD}_{g,n}=\overline{\DD}_{g-1,n+2}.	
\end{equation}
Conversely, assume that for a suitable choice of constant $E(g,n)$ --and the subsequent choice of Quillen metric determined by Definition \ref{definition:Quillen}-- the relations (\ref{equation:28bis_2})--(\ref{equation:28bis_3}) hold. An algebraic manipulation then shows that $E(0,3)$ determines $E(g,n)$. Furthermore, by Theorem A applied to $(\PP^{1}_{\Int};0,1,\infty)$, $E(0,3)$ coincides with the constant (\ref{equation:E}) (for $g=0$ and $n=3$), and so does $E(g,n)$. This explains the significance of Theorem A in the case $g=0$, $n=3$. We plan to deepen in these questions in the future.
\end{remark}
\begin{proof}[Proof of Theorem A]
Attached to $(\XX\rightarrow\BS;\sigma_{1},\ldots,\sigma_{n})$ there is a classifying morphism
\begin{displaymath}
	\C:\BS\rightarrow\SCM_{g,n}.
\end{displaymath}
For every embedding $\tau\in\Sigma$, $\XX_{\tau}=\XX\times_{\tau}\CC$ is smooth. Hence there is a commutative diagram
\begin{equation}\label{equation:29}
	\xymatrix{
		&\Spec\CC\ar[dd]^{\tau}\ar[rd]\ar@{-->}[r]^{\C_{\tau}}	&\SM_{g,n}\ar@{^{(}->}[d]\\
		&	&\SCM_{g,n}.\\
		&\BS\ar[ru]_{\C}	&
	}
\end{equation}
Let $\C^{\ast}\DD_{g,n}$ be the pull-back of the Mumford isomorphism $\DD_{g,n}$ by $\C$. From (\ref{equation:29}) we infer the equality $\tau^{\ast}\C^{\ast}\DD_{g,n}=\C_{\tau}^{\ast}\DD_{g,n}^{\circ}$ (recall that $\DD_{g,n}^{\circ}:=\DD_{g,n}\mid_{\SM_{g,n}}$). Therefore Theorem \ref{theorem:Mumford_isomorphism} and Theorem \ref{theorem:Mumford_isometry} altogether yield an isometry of hermitian line bundles on $\BS$
\begin{equation}\label{equation:30}
	\begin{split}
	\overline{\C^{\ast}\DD_{g,n}}:&\lambda(\omega_{\XX/\BS})_{Q}^{\otimes 12}\otimes\OO(\Delta_{\XX/\BS})^{-1}\otimes\psi_{W}\overset{\sim}{\longrightarrow}\\
	&\hspace{1.2cm}\langle\omega_{\XX/\BS}(\sigma_{1}+\ldots+\sigma_{n})_{\hyp},\omega_{\XX/\BS}(\sigma_{1.2}+\ldots+\sigma_{n})_{\hyp}\rangle\\
	&\hspace{1.2cm}\otimes\OO(C(g,n)).
	\end{split}
\end{equation}
The theorem is now obtained from (\ref{equation:30}) applying $\ac1$.
\end{proof}
We close this section with a further application of Theorem \ref{theorem:Mumford_isometry}: a  significant case of the Takhtajan-Zograf local index theorem \cite{ZT:ZT_metric_0}--\cite{ZT:ZT_metric}.
\begin{theorem}[Takhtajan-Zograf]\label{theorem:ZT}
Let $\omega_{WP}$, $\omega_{TZ}$ be the Weil-Petersson and Takhtajan-Zograf K\"ahler forms on $\SM_{g,n}^{\an}$, respectively. The following equality of differential forms on $\SM_{g,n}^{\an}$ holds:
\begin{equation}\label{equation:30_a}
	\c1(\lambda_{g,n;Q})=\frac{1}{12\pi^{2}}\omega_{WP}-\frac{1}{9}\omega_{TZ}.
\end{equation}
\end{theorem}
\begin{proof}
First of all we have the equality of differential forms on $\SM_{g,n}^{\an}$
\begin{equation}\label{equation:30_b}
	\c1(\overline{\kappa}_{g,n})=\frac{1}{\pi^{2}}\omega_{WP}.
\end{equation}
For a reference see \cite{Wolpert:hyperbolic} (case $n=0$) and \cite[Ch. 5]{GFM:thesis} (general case). Secondly, by \cite[Th. 5]{Wolpert:cusps}, there is another identity of differential forms on $\SM_{g,n}^{\an}$
\begin{equation}\label{equation:30_c}
	\c1(\psi_{g,n;W})=\frac{4}{3}\omega_{TZ}.
\end{equation}
The relation (\ref{equation:30_a}) is deduced by conjunction of Theorem \ref{theorem:Mumford_isometry} and (\ref{equation:30_b})--(\ref{equation:30_c}).
\end{proof}
\begin{remark}
In contrast with \cite[Fund. th., p. 278] {Weng}, our proof of (\ref{equation:30_a}) is new and does not require the work of Takhtajan-Zograf.
\end{remark}
\section{The special values $Z^{\prime}(Y(\Gamma),1)$ and $L(0,\SM_{\Gamma})$}\label{section:TheoremB}
The aim of this section is to proof Theorem B. The argument relies on Theorem A and a formula of Bost \cite{Bost} and K\"uhn \cite{Kuhn} for the arithmetic self-intersection number of $\omega_{X_{1}(p)/\QQ(\mu_{p})}(\text{cusps})_{\hyp}$.

Fix $K=\QQ(\mu_{p})\subset\CC$ the $p$-th cyclotomic field. Denote by $\iota$ the inclusion $K\subset\CC$ and by $\overline{\iota}$ its complex conjugate. Then $\mathcal{A}=(\Spec K,\Sigma:=\lbrace\iota,\overline{\iota}\rbrace,F_{\infty})$ is an arithmetic ring. The arithmetic Chow group $\ACH^{1}(\Spec K)$ --associated to $\mathcal{A}$-- comes equipped with an arithmetic degree map:
\begin{equation}\label{equation:35}
	\begin{split}
		\adeg:\ACH^{1}(\Spec K)&\longrightarrow\RR/\log|K^{\times}|\\
		[(0,\lbrace\lambda_{\iota},\lambda_{\overline{\iota}}\rbrace)]&\longmapsto 
		[\frac{1}{2}(\lambda_{\iota}+\lambda_{\overline{\iota}})]=[\lambda_{\iota}]\,\,\,
		(\text{since}\,\,\,\lambda_{\iota}=\lambda_{\overline{\iota}}).
	\end{split}
\end{equation}
If $\overline{L}=(L,\|\cdot\|)$ is a metrized line over $\Spec K$, then the first arithmetic Chern class $\ac1(\overline{L})\in\ACH^{1}(\Spec K)$ is the class of $(0,\lbrace\log\|s_{\sigma}\|_{\sigma}^{-2}\rbrace_{\sigma\in\Sigma})$, for any non-vanishing section $s$ of $L$. Therefore, taking (\ref{equation:35}) into account,
\begin{equation}\label{equation:36}
	\adeg\ac1(\overline{L})=[\log\|s_{\iota}\|_{\iota}^{-2}],
\end{equation}
which is independent of the choice of $s$.

Let $\QQQ$ be the algebraic closure of $\QQ$ in $\CC$. We will write $|\QQQ^{\times}|$ for the group of norms of elements of $\QQQ^{\times}$.

The modular curve $X(\Gamma)$, for $\Gamma=\Gamma_{0}(p)$ or $\Gamma_{1}(p)$, and its cusps are defined over the number field $K$ \cite[Ch. 6.7]{Shimura:autom}. We still write $X(\Gamma)$ for a projective model over $K$. The notations $X_{0}(p)$ and $X_{1}(p)$ will also be employed.\footnote{Contrary to the usual conventions, here $X_{0}(p)$ is assumed to be defined over $K$, and not over $\QQ$.}
\begin{proposition}\label{proposition:thmB_1}
Let $p$ be a prime number for which $X(\Gamma)$ has genus $g\geq 1$. Then the equality
\begin{displaymath}
	\adeg\ac1(\lambda(\omega_{X(\Gamma)/K}),\|\cdot\|_{L^2})=-\log(\pi^{-2g}L(0,\SM_{\Gamma}))
\end{displaymath}
holds in $\RR/\log|\QQQ^{\times}|$.
\end{proposition}
\begin{proof}[(see also \cite{Ullmo})]
First recall that via the $q$-expansion, $H^{0}(X(\Gamma),\omega_{X(\Gamma)/K})$ gets identified with the space of weight 2 cusp forms in $S_{2}(\Gamma,\CC)$ whose Fourier series expansion at $\infty$ have coefficients in $K$ \cite[Th. 1.33]{Darmon}. Let us write $S_{2}(\Gamma,K)$ for this space. Notice that the set $\text{Prim}_{2}(\Gamma)\subset S_{2}(\Gamma,K)\otimes_{K}\QQQ$ is a $\QQQ$-basis. For every $f\in\text{Prim}_{2}(\Gamma)$, write $\omega_{f}\in H^{0}(X(\Gamma),\omega_{X(\Gamma)/K})\otimes_{K}\CC$ for the corresponding differential form. Viewing $\omega_{f}$ as a holomorphic form, we find the relation
\begin{displaymath}
	4\pi\langle f,f\rangle=\frac{i}{2\pi}\int_{X(\Gamma)}\omega_{f}\wedge\overline{\omega_{f}},
\end{displaymath}
where $\langle f,f\rangle$ is the Petersson square norm of $f$. Indeed, $\omega_{f}$ pulls back to the $\Gamma$ invariant tensor $2\pi i f(z)dz$ on $\HH$. By the definition of the $L^{2}$ metric $\|\cdot\|_{L^{2}}$ (see Section \ref{section:conventions}) and (\ref{equation:36}), we infer the equality in $\RR/\log|\QQQ^{\times}|$
\begin{equation}\label{equation:37}
	\adeg\ac1(\lambda(\omega_{X(\Gamma)/K}),\|\cdot\|_{L^2})=-\log(\pi^{g}\prod_{f\in\text{Prim}_{2}(\Gamma)}\langle f,f\rangle).
\end{equation}
By \cite[Th. 5.1]{Hida} we have
\begin{equation}\label{equation:38}
	\prod_{f\in\text{Prim}_{2}(\Gamma)}\langle f,f\rangle\sim_{\QQQ^{\times}_{+}}\pi^{-3g}\prod_{f\in\text{Prim}_{2}(\Gamma)}L(2,\text{Sym}^{2}f,\overline{\chi_{f}}).
\end{equation}
The equations (\ref{equation:-1}) (see the Introduction), (\ref{equation:37}) and (\ref{equation:38}) altogether lead to the conclusion.
\end{proof}
\begin{lemma}\label{lemma:thmB}
Suppose that $p$ is a prime number congruent to $11$ modulo $12$. The following assertions hold:

i. the natural morphism $\varphi:X_{1}(p)\rightarrow X_{0}(p)$ is unramified as a morphism of $K$-schemes;

ii. $\varphi$ induces an isometry of pre-log-log hermitian line bundles
\begin{displaymath}
	\varphi^{\ast}(\omega_{X_{0}(p)/K}((0)+(\infty))_{\hyp})\overset{\sim}{\longrightarrow}\omega_{X_{1}(p)/K}(\cusps)_{\hyp};
\end{displaymath}

iii. if $\widetilde{\sigma}\in X_{1}(p)(K)$ is a cusp lying over $\infty\in X_{0}(p)(K)$ (resp. $0$), then $\varphi$ induces an isometry of hermitian line bundles
\begin{displaymath}
	\sigma_{\infty}^{\ast}(\omega_{X_{0}(p)/K})_{W}\overset{\sim}{\longrightarrow}\widetilde{\sigma}^{\ast}(\omega_{X_{1}(p)/K})_{W},
\end{displaymath}
where $\sigma_{\infty}$ is the section $\infty$ (resp. $\sigma_{0}$).
\end{lemma}
\begin{proof}
The first assertion is derived from \cite[Ch. 2, Sec. 2, Table I]{Mazur}. Then properties \textit{ii}--\textit{iii} are easily checked as a consequence of \textit{i}.
\end{proof}
\begin{proposition}\label{proposition:thmB_2}
Let $p\geq 11$ be a prime number. Assume that $p\equiv 11\mod 12$ whenever $\Gamma=\Gamma_{0}(p)$. Let $\psi_{W}$ be the tensor product of the cotangent bundles at the cusps of $X(\Gamma)_{K}$, endowed with the Wolpert metric. Then the equality
\begin{displaymath}
	\adeg\ac1(\psi_{W})=0
\end{displaymath}
holds in $\RR/\log|K^{\times}|$.
\end{proposition}
\begin{proof}
We begin with the more delicate case $\Gamma=\Gamma_{0}(p)$. Let $\sigma_{0}$ and $\sigma_{\infty}$ be the sections of $X_{0}(p)$ induced by the cusps $0,\infty$, respectively. We first show that $\adeg\sigma_{\infty}^{\ast}(\omega_{X_{0}(p)/K})=0$. Let $f\in S_{2}(\Gamma_{0}(p),K)$ be a cusp form whose $q$-expansion has leading coefficient $a_{1}\neq 0$. It exists because $\text{Prim}_{2}(\Gamma)$ is a base of $S_{2}(\Gamma_{0}(p),K)\otimes_{K}\QQQ$. Let $\theta$ be the global section of $\omega_{X_{0}(p)/K}$ corresponding to $f$ via the $q$-expansion. We claim that
\begin{equation}\label{equation:39_1}
	\|\sigma_{\infty}^{\ast}\theta\|_{W,\infty,\iota}=|a_{1}|\in |K^{\times}|.
\end{equation}
It shall be emphasized that (\ref{equation:39_1}) is not a formal consequence of the definition of the Wolpert metric: since $\Gamma_{0}(p)$ has elliptic fixed points, the hyperbolic metric on $\HH$ does not descend to the hyperbolic metric of $X_{0}(p)_{\CC}\setminus \lbrace 0,\infty\rbrace$.\footnote{The hyperbolic metric on $X_{0}(p)_{\CC}\setminus\lbrace 0,\infty\rbrace$ is smooth at the image of the elliptic fixed points, while the \textquotedblleft descended\textquotedblright\, metric is not \cite[Par. 4.2]{Kuhn}.} However we may pull-back $\theta$ to $X_{1}(p)$, where the hyperbolic metric on $X_{1}(p)_{\CC}\setminus\lbrace\text{cusps}\rbrace$ is obtained by descend from $\HH$. By the lemma we have
\begin{displaymath}
	\|\sigma_{\infty}^{\ast}\theta\|_{W,\infty,\iota}=\|\widetilde{\sigma}_{\infty}^{\ast}\varphi^{\ast}\theta\|_{W,\infty,\iota},
\end{displaymath}
where $\widetilde{\sigma}_{\infty}$ is the section of $X_{1}(p)$ induced by the cusp $\infty\in X_{1}(p)$. The restriction of the differential form $(\varphi^{\ast}\theta)_{\CC}$ to $X_{1}(p)_{\CC}\setminus\lbrace\text{cusps}\rbrace$ lifts to the upper half plane as
\begin{displaymath}
	f(z)=\sum_{n\geq 1}a_{n}q^{n}\frac{dq}{q},\,\,\,q=e^{2\pi i z}.
\end{displaymath}
By the very definition of the Wolpert metric we find
\begin{displaymath}
	\|\widetilde{\sigma}_{\infty}^{\ast}\varphi^{\ast}\theta\|_{W,\infty,\iota}=|a_{1}|,
\end{displaymath}
thus proving the claim. Notice that by (\ref{equation:36}), equation (\ref{equation:39_1}) entails
\begin{equation}\label{equation:39_2}
	\adeg\sigma_{\infty}^{\ast}(\omega_{X_{0}(p)/K})_{W}=[\log|a_{1}|^{-2}]=0\,\,\,\text{in}\,\,\,\RR/\log|K^{\times}|.
\end{equation}
Now we turn our attention to $\adeg\sigma_{0}^{\ast}(\omega_{X_{0}(p)/K})_{W}$. Consider the Atkin-Lehner involution $w_{p}:X_{0}(p)\rightarrow X_{0}(p)$ (see \cite[IV, (3.16)]{DR} and \cite[Ch. 2, Sec. 6, Par. 1]{Mazur}). The pull-back $w_{p}^{\ast}\theta$ is a global section of $\omega_{X_{0}(p)/K}$. From \cite[VII, (3.18)]{DR} and \cite[Ch. 2, Sec. 6, Par. 1]{Mazur}, the first coefficient of the $q$-expansion of $w_{p}^{\ast}\theta$ at the cusp $0$ is $a_{1}/p$. The same argument as above shows
\begin{displaymath}
	\|\sigma_{0}^{\ast}w_{p}^{\ast}\theta\|_{W,0,\iota}=|a_{1}|/p\in |K^{\times}|
\end{displaymath}
and hence
\begin{equation}\label{equation:39_3}
	\adeg\sigma_{0}^{\ast}(\omega_{X_{0}(p)/K})_{W}=0\,\,\,\text{in}\,\,\,\RR/\log|K^{\times}|.
\end{equation}
Equations (\ref{equation:39_2}) and (\ref{equation:39_3}) lead to the conclusion.

The argument for $\Gamma=\Gamma_{1}(p)$ is analogous. A comment is in order: to produce a global section of $\omega_{X_{1}(p)/K}$ not vanishing at a prescribed cusp, besides the Atkin-Lehner involution we need the diamond operators $\langle d\rangle:X_{1}(p)\rightarrow X_{1}(p)$, for $d\in(\Int/p\Int)^{\times}/\lbrace\pm 1\rbrace$.\footnote{The diamond operators constitute the Galois group of $X_{1}(p)\rightarrow X_{0}(p)$. This morphism is unramified at the cusps \cite[Sec. 2]{Mazur}. Hence the group generated by the Atkin-Lehner involution and the diamond operators acts transitively on the cusps of $X_{1}(p)$.} This completes the proof.
\end{proof}
\begin{proposition}\label{proposition:thmB_3}
Let $p\geq 11$ be a prime number and suppose that $p\equiv 11\mod 12$ whenever $\Gamma=\Gamma_{0}(p)$. The equality
\begin{displaymath}
		\adeg\pi_{\ast}(\ac1(\omega_{X(\Gamma)/K}(\cusps)_{\hyp})^{2})=
		4[\Gamma(1):\Gamma](2\zeta^{\prime}(-1)+\zeta(-1))
\end{displaymath}
holds in $\RR/\log|\QQQ^{\times}|$.
\end{proposition}
\begin{proof}
We begin with $\Gamma=\Gamma_{1}(p)$. By \cite[Tab. 10.13.9.1, Th. 10.13.11]{Katz-Mazur} there is a Kodaira-Spencer isomorphism
\begin{displaymath}
	\text{KS}:\omega_{\text{mod}}^{\otimes 2}\overset{\sim}{\longrightarrow}\omega_{X_{1}(p)/K}(\text{cusps}),
\end{displaymath}
where $\omega_{\text{mod}}^{\otimes 2}$ is the sheaf of weight 2 modular forms on $X_{1}(p)/K$. Equip $\omega_{\text{mod}}^{\otimes 2}$ with the hermitian structure $\|\cdot\|_{\text{mod}}=\text{KS}^{\ast}\|\cdot\|_{\hyp}$. A theorem of Bost and K\"uhn \cite[Th. 6.1, Cor. 6.2 and Rem. 6.3 a)]{Kuhn} provides
\begin{equation}\label{equation:40}
	\begin{split}
		\adeg\pi_{\ast}(\ac1(\omega_{\text{mod}}^{\otimes 2},\|\cdot\|_{\text{mod}})^{2})=&\\
		\sharp\lbrace\iota,\overline{\iota}\rbrace\cdot2[\Gamma(1):\Gamma_{1}(p)]&(2\zeta^{\prime}(-1)+\zeta(-1))
	\end{split}
\end{equation}
in $\RR/\log|\QQQ^{\times}|$.\footnote{The discussion of \cite[Par. 4.14]{Kuhn} shows that  $\|\cdot\|_{\text{mod}}$ is $1/\sqrt{2}$ times the Petersson metric (or $L^{2}$ metric) used by Bost and K\"uhn. This disagreement contributes to 0 in the arithmetic self-intersection number (\ref{equation:40}), whose value is in $\RR/\log|\QQQ^{\times}|$.}

We now focus on $\Gamma_{0}(p)$, $p\equiv 11\mod 12$. Let $d$ denote the degree of $\varphi:X_{1}(p)\rightarrow X_{0}(p)$. By Lemma \ref{lemma:thmB} and the functoriality of the arithmetic self-intersection numbers, we have
\begin{equation}\label{equation:41}
	\begin{split}
	\adeg&\pi_{\ast}(\ac1(\omega_{X_{0}(p)/K}((0)+(\infty))_{\hyp})^{2})=\\
	&\frac{1}{d}\adeg\pi_{\ast}(\ac1(\omega_{X_{1}(p)/K}(\text{cusps})_{\hyp})^{2}).
	\end{split}
\end{equation}
Because $d=[\Gamma_{0}(p):\Gamma_{1}(p)]$, the claim follows from (\ref{equation:40}) and (\ref{equation:41}).
\end{proof}
\begin{proof}[Proof of Theorem B]
The result follows from Theorem A, the relation (\ref{equation:36}), Proposition \ref{proposition:thmB_1}--\ref{proposition:thmB_3} and the identity
\begin{displaymath}
	\exp(\zeta^{\prime}(-1))=2^{-1/36}\pi^{1/6}\Gamma_{2}(1/2)^{-2/3}
\end{displaymath}
(see \cite[App.]{Voros}).
\end{proof}
\begin{remark}
We expect that $\sim_{\QQQ^{\times}}$ can be refined to an equality. A possible approach would be to work with the regular models of Deligne-Rapoport \cite{DR} and, at least in the case $\Gamma=\Gamma_{0}(p)$, apply a theorem of Ullmo on the Faltings height of the Jacobian $J_{0}(p)$ \cite{Ullmo}. However, the necessary considerations have not been effected here: the singular fiber of $\XX_{0}(p)/\Int$, together with the sections $0$, $\infty$, is not a pointed stable curve \cite[Ch. II, Sec. 1 and App.]{Mazur}.
\end{remark}
\textbf{Acknowledgements}

The starting point of this work was the suggestion of my thesis advisors J.-B. Bost and J. I. Burgos of studying the Arakelov geometry of the moduli spaces $\SCM_{g,n}$. I am deeply indebted to them for the many hours of mathematical discussions on the topic of this work, as well as for their constant insistence and encouragement. I also thank J.-M. Bismut, G. Chenevier, T. Hahn, M. Harris, S. Keel, J. Kramer, U. K\"uhn, X. Ma, V. Maillot, C. Soul\'e, L. Takhtajan and E. Ullmo for discussions on this and related topics. Special thanks go to Scott Wolpert, who kindly shared his ideas with me. A substantial part of this work was done while the author was visiting the Universit\`a degli Studi di Padova, with the support of the European Commission's \textquotedblleft Marie Curie early stage researcher\textquotedblright\, programme (6th Framework Programme RTN \textit{Arithmetic Algebraic Geometry}). I thank these institutions for their hospitality and financial support.
\bibliographystyle{amsplain}

\end{document}